\documentclass[10pt]{article}
\usepackage{amsmath,amsthm,amssymb,graphicx}
\usepackage{color}

\usepackage{geometry}
\usepackage{tikz}
\usetikzlibrary{arrows}

\newtheorem{theorem}{Theorem}[section]
\newtheorem{lemma}[theorem]{Lemma}

\newtheorem{proposition}[theorem]{Proposition}
\newtheorem{definition}[theorem]{Definition}

\newtheorem{corollary}[theorem]{Corollary}

\newtheorem*{theorem*}{Theorem}

\theoremstyle{definition}
\newtheorem{example}[theorem]{Example}

\theoremstyle{remark}
\newtheorem{remark}[theorem]{Remark}

\newcommand{\lin}{BL}
\newcommand{\set}{\operatorname{set}}
\newcommand{\F}{\mathcal{F}}
\newcommand{\C}{\mathcal{C}}

\newcommand{\m}{\mathfrak{m}}

\newcommand{\rmod}{\operatorname{Mod}_R}
\newcommand{\smod}{\operatorname{Mod}_S}
\newcommand{\crep}{\operatorname{Rep}_R(\C)}
\newcommand{\pair}{\mathcal{P}}
\newcommand{\repS}{\operatorname{Rep}_S}
\newcommand{\pathidealn}{I_{P_n}}
\newcommand{\pathn}{P_n}
\newcommand{\maxfam}{\F_{\m}}
\newcommand{\infring}{S_{\infty}}
\newcommand{\infmod}{\operatorname{Mod}_{\infring}}
\newcommand{\submodinf}{s_t^{\infty}}
\newcommand{\syzmodinf}{\sigma_t^{\infty}}
\newcommand{\edgecat}{\operatorname{CellRes}_E(n)}
\newcommand{\edgecatbl}{\operatorname{CellRes}_E(n+\frac{n(n-1)}{2})}

\begin{document}

\title{Families of cellular resolutions, their syzygies, and stability}

\author{Laura Jakobsson}

\date\today

\maketitle
\begin{abstract}
We study families of cellular resolutions by looking at them as a category and applying tools from representation stability. We obtain sufficient conditions on the structure of the family to have a noetherian representation category and apply this to concrete examples of families. In the study of syzygies we make use of defining the syzygy module as a representation and find conditions for the finite generation of this representation. We then show that many families of cellular resolutions coming from powers of ideals satisfy these conditions and have finitely generated syzygies, including the maximal ideals and edge ideals of paths. 
\end{abstract}
\section{Introduction}
Cellular resolutions are a powerful construction for resolving modules given by monomial ideals and they contain a lot of structure (\cite{bps},\cite{def}). 
 Computational evidence on cellular resolutions have long suggested that in some families we have finitely generated syzygies given by some finite number of resolutions in the family; however, there has not been a satisfactory proof of this. In this paper we study families of cellular resolutions from a categorical perspective and by using representations of categories that is motivated by computational results on syzygies. Categorical representation stability, in particular the tools that were introduced by Sam and Snowden in \cite{SS}, has been useful while studying noetherianity and finite generation of representations.
 Using tools from representation stability we establish sufficient conditions for families of cellular resolutions to have finitely generated syzygies.  
 The main idea is to define syzygies as a representation of the family and then show finite generation for this representation using noetherianity and covering of the cell complexes. This method allows us to use families with non-minimal cellular resolutions to study the syzygies.
 
The main result of this paper is on the conditions when a family of cellular resolutions has finitely generated syzygies.

\begin{theorem*}
If $\F$ is a family of cellular resolutions with noetherian representation category $\mathrm{Rep}_S(\F)$ such that the cell complex supporting $F_i$ is covered by the cell complexes supporting $F_{j}$, $j<i$, for all $i$ large enough.
Then the syzygy representation $\sigma_t$ is finitely generated for all $t$.
\end{theorem*}

Other than being able to show that certain families have finitely generated syzygies, using categorical representations to study cellular resolutions can give new insights and it equips them with a structure. In particular it seems to be suited to studying cellular resolutions of powers of ideals. We illustrate this by studying the specific examples of families, the powers of maximal ideals and edge ideals of paths are among the examples we prove to have finitely generated syzygies. 

The organisation of the paper is as follows. In Section 2 we cover the needed background material for cellular resolutions, their morphisms and the necessary tools from representation stability. In Section 3 we give the definition of linear families, detailed  example of powers of $I=(x,y,z)$, and the results on noetherianity and Gr\"obner properties.
Section 4 is devoted to studying the syzygy representation and its relation to the covering of cell complexes and contains the main theorem. 
Sections 5 and 6 cover explicit families including edge ideals of paths and maximal ideals.
Finally in Section 7 we suggest a way of dealing with a family of cellular resolutions where each resolution is over a different ring. This section addresses some of the limitations the setting of the earlier sections have had and we suggest an alternative way of approaching the family categorically that allows us to lift the results of the previous sections. 
\subsection*{Acknowledgements}
I would like to thank Alexander Engstr\"om for many helpful discussions and his guidance. 
\section{Preliminary results}
\subsection{Cellular resolutions}
In this section we define the needed notions of cellular resolutions and their category. More detailed information on cellular resolutions can be found in \cite{cca} and on the categorical aspects in \cite{me}. 
\begin{definition}
 A labelled cell complex $X$ is a regular CW-complex with monomial labels on the faces. The  vertices of $X$ have labels $\bf{x}^{\bf{a}_1},\bf{x}^{\bf{a}_2},\ldots,\bf{x}^{\bf{a}_r}$ where $\bf{a}_1,\bf{a}_2,\ldots,\bf{a}_r\in\mathbb{N}^n$. The faces $F$ of $X$ have the least common multiple of the monomial labels of the vertices it contains, $x^{\bf{a}_F}=\mathrm{lcm}\{x^{\bf{a}_v}:v\in F\}$. The label on the empty face is 1, i.e. $\bf{x}^{\bf{0}}$.
\end{definition}
\begin{definition} 
  The \emph{degree} of a face $F$ is the exponent vector $\bf{a}_F$ of the monomial label.
\end{definition}

\begin{definition}
Let $S(-\bf{a}_F)$ be the free $S$-module with a generator $F$ in degree $\bf{a}_F$. Then the \emph{cellular complex} $\mathcal{F}_X$ is given by $(\mathcal{F}_X)_i=\bigoplus_{\substack{F\in X\\ \mathrm{dim}F=i-1}}S(-\bf{a}_F)$ with a differential  $$\partial(F)=\sum_{G\subset F}\operatorname{sign}(G,F)x^{\bf{a}_F-\bf{a}_G}G.$$  
We call the chain complex $\mathcal{F}_X$ a \emph{cellular resolution} if it is acyclic, that is, $\mathcal{F}_X$ has non-zero homology only at degree 0.
\end{definition}

\begin{proposition}[\cite{cca}, Def 4.3]
The differentials in the  cellular complex can also be described by monomial matrices, with the columns and rows having the corresponding faces as labels and the scalar entries coming from the usual differential for reduced chain complex. The free $S$-modules of $\mathcal{F}_X$ are then the ones represented by the matrices. 
\end{proposition}

Another useful result for cellular resolutions makes use of order of vectors. If {\bf a} and {\bf b} are two vectors in $\mathbb{N}^n$, we have ${\bf a}\preceq{\bf b}$ if ${\bf b}-{\bf a}\in \mathbb{N}^n$. Let $X$ be a labelled cell complex, then we can define the subcomplex $X_{\preceq {\bf b}}$ to be the complex consisting of all the faces with labels $\preceq {\bf b}$. Then we have the following.
\begin{proposition}[\cite{cca}, Prop 4.5]
\label{bound}
The cellular free complex $\F_X$ supported on $X$ is a cellular resolution if and only if $X_{\preceq{\bf b}}$ is acyclic over {\bf k} for all ${\bf b}\in\mathbb{N}^n$.
\end{proposition}

For category-theoretic purposes one needs a morphism between two cellular resolutions. This was defined in \cite{me}.
Informally a morphism between two cellular resolutions is a pair of a chain map and a cellular map that do the same thing on the corresponding cells and generators of modules. More formally we get the definitions below.
\begin{definition}
\label{inducedlabelmap}
Let $g:X\rightarrow Y$ be a cellular map between two labelled cell complexes $X$ and $Y$ with label ideals $I$ and $J$ respectively. 
The set map $\varphi_g: I\rightarrow J$ is the map defined by the action of $g$, i.e. label $m_x\in I$ maps to $m_y\in J$ if and only if the face $x$ labelled with $m_x$ maps to the faces $y_1,\ldots, y_r$ labelled by $m_{y_1},\ldots,m_{y_r}$with $m_y=\mathrm{lcm}(m_{y_1},\ldots,m_{y_r})$ under $g$, and $m_x\in I$ maps to $0$ if and only if the face labelled by $m_x$ is not mapped to anything in $Y$. 
\end{definition}

\begin{definition}
We say that a cellular map $g:X\rightarrow Y$ is compatible with a chain map ${\bf f}:\mathcal{F}_X\rightarrow \mathcal{F}_Y$ if $f_0(x)=\varphi_g(x)$ for all $x\in I$, and $f_i$ maps the generator $e_x$, associated to face $x\in X$, in $\mathcal{F}_{X,i}$ to some linear combination of the generators $e_{y_i}$, $i\in\{1,2,\ldots,r\}$, associated to $y_i\in Y$ with the coefficients in $S$ if and only if $g$ maps $x$ to union of $y_1,y_2,\ldots, y_r$. 
\end{definition}
This definition of maps between cellular resolutions is very restrictive and balances the topological and algebraic properties. One thing to note is that the component $f_0$ in the chain map must be a $1\times1$ matrix, i.e. multiplication by some element of $S$.  However, this does leave us with plenty of maps including Morse maps coming from Morse theory and the change of orientation on the cell complex. 
\begin{definition}
\label{DEF}
 Define {\bf CellRes} to be the category given by:
\begin{itemize}
\item A class of objects consisting of cellular resolutions, supported on any regular CW-complex,
\item A set of morphisms for any pair of objects $\mathcal{F}_X$ and $\mathcal{G}_Y$ with individual maps given by the compatible pairs $({\bf f},f)$.
\end{itemize}
\end{definition}

The morphisms between cellular resolutions are an important part in most of the results relating to representation stability. Here we list a few observations and notions for them. The first definition concerns the terminology which we will use later in the paper.

\begin{definition}
Let $({\bf f},f)$ be a morphism of cellular resolutions. It is called a \emph{multiplication by monomial $m$} or a morphism corresponding to a multiplication  if $\varphi_f$ is a multiplication by a monomial $m$.
\end{definition}

\begin{proposition}
Embeddings in cell complexes supporting cellular resolutions give cellular morphisms.
\end{proposition}
\begin{proof}
Let $X$ and $Y$ be two labelled cell complexes supporting cellular resolutions $F_X$ and $F_Y$, respectively, and suppose that we have an embedding $f:X\rightarrow Y$. Since the embedding of labelled cell complexes respects the labelling we get that the map $\varphi_f$ is an identity map. Then we can find a chain map with $f_0=\varphi_f$ and choose the rest of the matrices with entries 1 or -1 based on where $f$ takes the higher dimensional cells. Then these form a compatible pair and we have a morphism of cellular resolutions corresponding to embeddings.
\end{proof}
\subsection{Representation stability}
\label{repstab}
In this section we review the concepts needed from representation stability as defined by Sam and Snowden \cite{SS}.  Let $R$ be a commutative noetherian ring. This assumption is not necessary but in our setting almost all rings are commutative polynomial rings with finitely many variables. Hence, they are also noetherian. Let $\rmod$ be the category of $R$ modules.

Throughout this section we assume the category $\C$ to be essentially small. Recall that this means the category $\C$ is equivalent to some small category or alternatively it is locally small and has small number of isomorphism classes as objects (assuming the axiom of choice). For more on the category theory definitions and theorems one can look at \cite{ml} for example. We want the category $\C$ to be of "combinatorial nature", which informally means objects are finite sets, possibly with some extra structures, and morphisms are functions with extra structure allowed. 
 
\begin{definition}
Let $\C$ be an essentially small category. A \emph{representation} or a $\C$-module over $R$ is a functor 
$$\C\rightarrow \rmod.$$
\end{definition}
The representations of $\C$ form a category denoted by $\crep$. This is an abelian functor category with the morphisms between representations given by natural transformations. 

Next we want to look at some definitions related to the properties of individual representations (or modules). Let $M$ be a representation of $\C$.  A subrepresentation $N$ of $M$ is a subfunctor of $M$, that is 
\begin{definition}
Let $M$ be a representation of $\C$. An \emph{element} of $M$ is an element of $M(x)$ for some $x\in \C$. 
\end{definition}
Having defined an element one can then talk about the generating sets for representations. 
\begin{definition}
\label{fgrep}
Let $S$ be any set of elements of $M$. The smallest subrepresentation of $M$ containing $S$ is said to be \emph{generated by S}. The representation $M$ is said to be \emph{finitely generated} if it is generated by some finite set of elements.
\end{definition}

The following representation is one of the main tools used to study noetherianity for the representations.
\begin{definition}[\cite{SS}]
The \emph{principal projective} representation for an element $x$ is the functor $P_x$ given by $P_x(y)=R[\textrm{Hom}(x,y)]$.
\end{definition}
\begin{remark}
In the paper of Sam and Snowden \cite{SS} they do not explicitly give the morphism part of the principal projective. The natural choice of maps between the Hom sets in the principal projective are post compositions, so this gives then a morphism between $P_x(y)$ and $P_x(z)$ if we have a morphism $f: y\rightarrow z$.
\end{remark}
An important fact about the principal projectives is that a representation of $\C$ is finitely generated if and only if it is a quotient of  a finite direct sum of principal projectives. 

\begin{definition}
Let $M\in\crep$. We say $M$ is \emph{noetherian} if every ascending chain of subobjects stabilises, or equivalently every subrepresentation is finitely generated. 

The category $\crep$ is \emph{noetherian} if every finitely generated representation in it is noetherian.
\end{definition}

Next we have the following result.
\begin{proposition}[\cite{SS}, Prop 3.1.1.]
\label{noeth}
The category $\textrm{Rep}_R(\C)$ is noetherian if and only if every principal projective is noetherian.
\end{proposition}


One way to study representations is to use pullback functors, and we will use this in Section \ref{blideals}. Given a functor $\Phi:\C\rightarrow\C'$ there is a \emph{pullback functor} $\Phi^*:\operatorname{Rep}_R(\C')\rightarrow\crep$. The following finiteness property is particularly useful.

\begin{definition}[\cite{SS}, Def 3.2.1.]
Let $\Phi:\C\rightarrow\C'$ be a functor. Then $\Phi$ satisfies \emph{property (F)} if given any object $x\in\C'$ there exists finitely many $y_1,y_2,\ldots,y_n\in\C$ and morphisms $f_i:x\rightarrow\Phi(y_i)$ such that for any $y\in\C$ and any morphism $f:x\rightarrow\Phi(y)$ there exists a morphims $g:y_i\rightarrow y$ such that $f=\Phi(g)\circ f_i$.
\end{definition}

\begin{proposition}[\cite{SS},Prop 3.2.3.]
A functor $\Phi:\C\rightarrow\C'$ satisfies the property (F) if and only if $\Phi^*:\operatorname{Rep}_R(\C')\rightarrow\crep$ maps finitely generated objects to finitely generated objects. 
\end{proposition}

Finally, we move on to covering the definition of Gr\"obner bases for categories and Gr\"obner categories from Sam and Snowden \cite{SS}. Those will make an appearance  in Section \ref{grobnerfams}.
Let $S:\C\rightarrow Set$ denote a fixed functor to sets and let $S_x:\C\rightarrow Set$ be the functor given by $S_x(y)=\operatorname{Hom}(x,y)$.
A \emph{principal subfunctor} is a subfunctor of $S$ generated by a single element. 
\begin{definition}
The poset $|S|$ is the set of principal subfunctors of $S$ that is partially ordered by reverse inclusion.
\end{definition}
Let $P$ denote the free module $R[S]$ and write $e_f$ for the element of $P(x)$ corresponding to $f\in S(x)$. An element of $P(x)$ is \emph{monomial} if it is of the form $\lambda e_f$ for some $\lambda\in R$. A subrepresentation $M$ is monomial if it is spanned by the monomials it contains. 

To define Gr\"obner basis we need a concept of initial representations and terms. 
The functor $S$ is \emph{orderable} if there is a choice of well-order on each $S(x)$ such that the induced map $S(x)\rightarrow S(y)$ is strictly order preserving for every $x\rightarrow y$.

Suppose $S$ has ordering $\preceq$ on it. Then the \emph{initial term} of an object $\alpha\in P(x)$ is $\operatorname{init}(\alpha)=\lambda_g e_g$, where $g=\operatorname{max}_{\preceq}\{f|\lambda_f\neq0\}$ and $\alpha$ is a direct sum of monomials. Let $M$ be a subfunctor of $P$. The \emph{initial representation} of $M$ consists of $\operatorname{init}(M)(x)$ that is the $R$-span of $\operatorname{init}(\alpha)$ for $\alpha\neq0\in M(x)$.

\begin{definition}
Let $M$ be a subrepresentation of $P$. A set of elements $G$ of $M$ is a Gr\"obner basis of $M$ is $\{\operatorname{init}(\alpha)|\alpha\in G\}$ generates $\operatorname{init}(M)$.
\end{definition}

\begin{theorem}[\cite{SS}, Thm 4.2.4.]
Let $S$ be orderable and $|S|$ be noetherian. Then every subrepresentation of $P$ has finite Gr\"obner basis. In particular, $P$ is a noetherian object of $\crep$.
\end{theorem}

\begin{definition}
Let $\C$ be an essentially small category.  Then $\C$ is called \emph{Gr\"obner} if for all $x\in \C$ the functor $S_x$ is orderable and the poset $|S_x|$ is noetherian.

We say that $\C$ is \emph{quasi-Gr\"obner} if there exists some Gr\"obner category $\C'$ such that there is a functor $\Phi:\C'\rightarrow\C$ that is essentially surjective and satisfies property (F). 
\end{definition}

\begin{theorem}[\cite{SS}, Thm 4.3.2.]
Let $\C$ be quasi-Gr\"obner, then $\crep$ is noetherian.
\end{theorem}
In the case the category is directed as well as small we can use the following proposition to determine if it is Gr\"obner.
First note that an \emph{admissible order} is a well-order on a set that also satisfies if we have any two elements $u\leq v$ then for any third element $t$, for which $ut$ and $vt$ make sense, we have $ut\leq vt$.
\begin{proposition}[\cite{SS}, Prop 4.3.4.]
\label{bettergrobner}
If C is a directed category, then as posets $|\C_x|\cong |S_x|$ for all objects $x$. In particular, $\C$ is Gr\"obner if and only if for all $x$ the set $|\C_x|$ admits an admissible order and is noetherian as a poset.
\end{proposition}

\section{Families of cellular resolutions and noetherian representations}
Cellular resolutions as a whole form a too big class of objects to study via representation stability. Thus we will restrict to the families of cellular resolutions that are essentially small, or small in many cases, and are interesting on their own as a restricted class of cellular resolutions. 

\begin{definition}
A family of cellular resolutions is an infinite sequence of cellular resolutions such that as a subcategory of {\bf CellRes} it is essentially small category. 
\end{definition}
\begin{remark}
From a representation stability point of view, any essentially small subcategory would suffice. If we have an indexed sequence then the family will have countably many objects, which clearly form a set. So in most cases the only possible cause for the family to not be essentially small are the morphisms. 
\end{remark}

The above definition of a family of cellular resolutions does not restrict the morphisms in any other way that on the essentially small part. In particular one can choose the cellular resolutions at random without needing any morphisms between them. However these kind of families are not the main interest of our study and we will restrict to ones with more structure. 
Note that if we do have any morphisms in the family the Hom sets must be sets and not a class for the essentially small condition to be satisfied.
A common way that we use in this paper is to restrict the morphisms to only one compatible pair for each chain map. 
The following example presents some explicit examples of families of cellular resolutions. 

\begin{example}

\begin{itemize}
\item[ ]
\item The constant family where each of the resolutions is the same. 
\item $S/I^n$ where $I$ is any monomial ideal and the consecutive maps are multiplications by generators of $I$.
\item Cellular resolutions of edge ideals of paths with embeddings as maps between them.
\end{itemize}
\end{example}

\begin{definition}
\label{linearfam}
Let $\F: F_1\rightarrow F_2\rightarrow\ldots\rightarrow F_i\rightarrow\ldots $ be a family of cellular resolutions. We say that $\F$ is \emph{linear} if there is at least one morphism $f_{i,i+1}: F_i\rightarrow F_{i+1}$ between consecutive cellular resolutions, and  the other  morphisms are compositions of those, i.e  for any $f_{i,i+k}: F_i\rightarrow F_{i+k}$ there exists some consecutive morphisms such that $f_{i,i+k}=f_{i+k-1,i+k}\circ f_{i+k-2,i+k-1}\circ\ldots\circ f_{i+1,i+2}\circ f_{i,i+1}$, except possibly selfmaps $F_i\rightarrow F_i$. 
\end{definition}
\begin{remark}
Note that the definition does not require the decomposition of the map $f_{i,i+k}$ to be unique. 
\end{remark}

A common example of linear family is the family of cellular resolution consisting of powers of an ideal $I$. Families consisting of resolutions corresponding to the powers of some ideal can be defined to have multiplication maps that correspond to the monomials in the ideal, and so the composition part of the linear family is easily satisfied. Of course the freedom of choice allows one to choose the morphisms such that any power family can be non-linear, so even with families of powers there is need to clarify the morphisms and check linearity in all cases. 

\subsection{Example: powers of I=(x,y,z)}
\label{triangleex}
We begin with an explicit example of a family of cellular resolutions that has finitely generated syzygies where this can be shown with the help of representation stability.
The propositions presented in this example are special cases of theorems and propositions in the later sections. 

Let $S=k[x,yz]$ and let $I=(x,y,z)$ be an ideal. We are interested in studying the syzygies of the modules $S/I^n$ for positive integer $n$. 

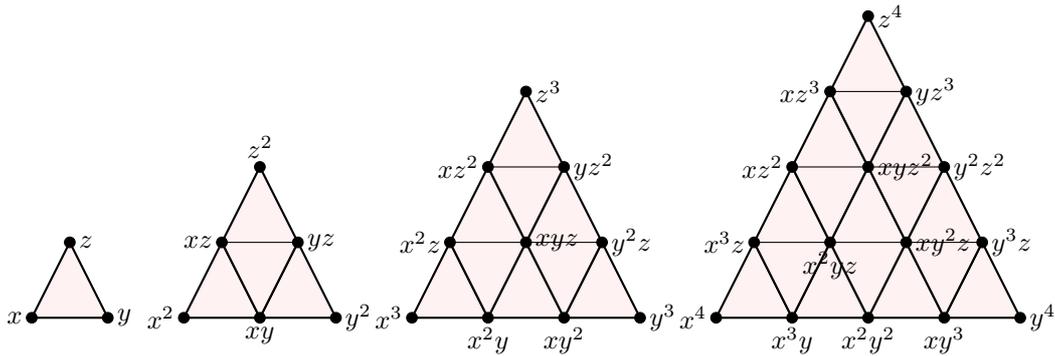
\begin{figure}
\begin{tikzpicture}
\filldraw[color=black, fill=red!5,  thick](-3,0) -- (-2.5,1) -- (-2,0) -- cycle;
\filldraw [black] (-3,0) circle (2pt) node[anchor=east] {$x$};
\filldraw [black] (-2,0) circle (2pt) node[anchor=west] {$y$};
\filldraw [black] (-2.5,1) circle (2pt) node[anchor=west] {$z$};

\filldraw[color=black, fill=red!5,  thick](-1,0) -- (0,0) -- (-0.5,1) -- cycle;
\filldraw[color=black, fill=red!5,  thick](0,0) -- (0.5,1) -- (1,0) -- cycle;
\filldraw[color=black, fill=red!5,  thick](0,0) -- (0.5,1) -- (0,2) --  (-0.5,1) -- cycle;
\draw (0.5,1) -- (-0.5,1);
\filldraw [black] (-1,0) circle (2pt) node[anchor=east] {$x^2$};
\filldraw [black] (1,0) circle (2pt) node[anchor=west] {$y^2$};
\filldraw [black] (0,0) circle (2pt) node[anchor=north] {$xy$};
\filldraw [black] (-0.5,1) circle (2pt) node[anchor=east] {$xz$};
\filldraw [black] (0.5,1) circle (2pt) node[anchor=west] {$yz$};
\filldraw [black] (0,2) circle (2pt) node[anchor=south] {$z^2$};

\filldraw[color=black, fill=red!5,  thick](2,0) -- (3,0) -- (2.5,1) -- cycle;
\filldraw[color=black, fill=red!5,  thick](3,0) -- (3.5,1) -- (4,0) -- cycle;
\filldraw[color=black, fill=red!5,  thick](3,0) -- (3.5,1) -- (3,2) --  (2.5,1) -- cycle;
\draw (3.5,1) -- (2.5,1);
\filldraw[color=black, fill=red!5,  thick](4,0) -- (4.5,1) -- (5,0) -- cycle;
\filldraw[color=black, fill=red!5,  thick](4,0) -- (4.5,1) -- (4,2) --  (3.5,1) -- cycle;
\filldraw[color=black, fill=red!5,  thick](3.5,1) -- (4,2) -- (3.5,3) --  (3,2) -- cycle;
\draw (3.5,1) -- (4.5,1);
\draw (4,2)--(3,2);

\filldraw [black] (2,0) circle (2pt) node[anchor=east] {$x^3$};
\filldraw [black] (4,0) circle (2pt) node[anchor=north] {$xy^2$};
\filldraw [black] (3,0) circle (2pt) node[anchor=north] {$x^2y$};
\filldraw [black] (2.5,1) circle (2pt) node[anchor=east] {$x^2z$};
\filldraw [black] (3.5,1) circle (2pt) node[anchor=west] {$xyz$};
\filldraw [black] (3,2) circle (2pt) node[anchor=east] {$xz^2$};
\filldraw [black] (3.5,3) circle (2pt) node[anchor=west] {$z^3$};
\filldraw [black] (4,2) circle (2pt) node[anchor=west] {$yz^2$};
\filldraw [black] (4.5,1) circle (2pt) node[anchor=west] {$y^2z$};
\filldraw [black] (5,0) circle (2pt) node[anchor=west] {$y^3$};

\filldraw[color=black, fill=red!5,  thick](6,0) -- (7,0) -- (6.5,1) -- cycle;
\filldraw[color=black, fill=red!5,  thick](7,0) -- (7.5,1) -- (8,0) -- cycle;
\filldraw[color=black, fill=red!5,  thick](7,0) -- (7.5,1) -- (7,2) --  (6.5,1) -- cycle;
\draw (7.5,1) -- (6.5,1);
\filldraw[color=black, fill=red!5,  thick](8,0) -- (8.5,1) -- (9,0) -- cycle;
\filldraw[color=black, fill=red!5,  thick](8,0) -- (8.5,1) -- (8,2) --  (7.5,1) -- cycle;
\filldraw[color=black, fill=red!5,  thick](7.5,1) -- (8,2) -- (7.5,3) --  (7,2) -- cycle;
\draw (7.5,1) -- (8.5,1);
\draw (8,2)--(7,2);
\filldraw[color=black, fill=red!5,  thick](9,0) -- (9.5,1) -- (10,0) -- cycle;
\filldraw[color=black, fill=red!5,  thick](9,0) -- (9.5,1) -- (9,2) --  (8.5,1) -- cycle;
\filldraw[color=black, fill=red!5,  thick](8.5,1) -- (9,2) -- (8.5,3) --  (8,2) -- cycle;
\draw (8.5,1) -- (9.5,1);
\draw (9,2)--(8,2);
\filldraw[color=black, fill=red!5,  thick](8,2) -- (8.5,3) -- (8,4) --  (7.5,3) -- cycle;
\draw (8.5,3)--(7.5,3);
\filldraw [black] (6,0) circle (2pt) node[anchor=east] {$x^4$};
\filldraw [black] (8,0) circle (2pt) node[anchor=north] {$x^2y^2$};
\filldraw [black] (7,0) circle (2pt) node[anchor=north] {$x^3y$};
\filldraw [black] (6.5,1) circle (2pt) node[anchor=east] {$x^3z$};
\filldraw [black] (7.5,1) circle (2pt) node[anchor=north] {$x^2yz$};
\filldraw [black] (7,2) circle (2pt) node[anchor=east] {$xz^2$};
\filldraw [black] (7.5,3) circle (2pt) node[anchor=east] {$xz^3$};
\filldraw [black] (8,2) circle (2pt) node[anchor=west] {$xyz^2$};
\filldraw [black] (8.5,1) circle (2pt) node[anchor=west] {$xy^2z$};
\filldraw [black] (9,0) circle (2pt) node[anchor=north] {$xy^3$};
\filldraw [black] (8,4) circle (2pt) node[anchor=west] {$z^4$};
\filldraw [black] (8.5,3) circle (2pt) node[anchor=west] {$yz^3$};
\filldraw [black] (9,2) circle (2pt) node[anchor=west] {$y^2z^2$};
\filldraw [black] (9.5,1) circle (2pt) node[anchor=west] {$y^3z$};
\filldraw [black] (10,0) circle (2pt) node[anchor=west] {$y^4$};

\end{tikzpicture}
\caption{The first four cell complexes in the family of cell complexes supporting the non-minimal resolutions.}
\label{nonminpic1}
\end{figure}

Let $X_n$ denote the labelled cell complex with labels from $I^n$. The cell complex $X_n$ is made of a triangle that is subdivided into  $n^2$ smaller triangles.
Let the triangle bounded by the vertices $x^n,y^n$ and $z^n$ be called the \emph{outer triangle.} We refer to the vertices of the cell complex $X_n$ by the labels. 
Consider the family of resolutions supported on the cell complexes $X_n$.  Figure \ref{nonminpic1} shows the first few labelled cell complexes in the family. This family consists of non-minimal cellular resolutions. We assume that the orientation is fixed.
Fixing the orientation of the cell complexes means that we only have one cell complex for each set of labels. 

Next we want to study the possible morphisms between the cellular resolutions supported on $X_n$. 
Recall that in \cite{me} the morphisms are defined as compatible pairs of cellular maps and chain maps. In the case of the family supported on the cell complexes in Figure \ref{nonminpic1} we can approach the morphism by studying the possible cellular maps. Since we are interested in the syzygies, we do not need the information of all different maps compatible with a chain map. Thus it suffices to find just one  cellular map for each compatible chain map.  In practice this means we only focus on which cells are mapped to which cells, and we do not care about how it is mapped as topological map.

Recall that in the induce label map $\varphi_g$ can be found in the Definition \ref{inducedlabelmap}.

\begin{proposition}
The possible cellular maps $g:X_n\rightarrow X_{n+1}$, such that they are a component of a cellular resolution morphism, are the ones inducing label maps by multiplication by a variable.
\end{proposition}
\begin{proof}
By definition of compatibility the cellular map $g$ must induce a map $\varphi_g$ between the label ideals. 
If we have a cellular map where any cell maps into a lower-dimensional cell, the induced label map would not give a map between ideals. 
Thus the cellular maps have to map cells to the same dimensional cells.

 We have three easily found maps for the maps taking cells to same dimensional cells.  The cellular map $g$ taking vertex $m$ to the vertex $mx$ defines a label map $\varphi_g$ that is multiplication by $x$. A simple computation shows that the cellular map $g$ is compatible with the chain map ${\bf f}$ between the resolutions supported on $X_n$ and $X_{n+1}$ such that $f_0=\varphi_g$.  This cellular map $g$ gives a morphism of cellular resolutions. 
Swapping the variable $x$ to either $y$ or $z$ gives a similar map, the multiplication is just by a different variable and we still get a morphism of cellular resolutions. 

In the case $n\geq 2$, we do not have other maps that do not involve permutations of variables. 
If we are in the case $X_1$ mapping to $X_2$ there is the cellular map taking $X_1$ to the central triangle of $X_2$. 
This gives us three possible label maps on the vertices which are 
$$\begin{array}{ccclccclccc}
x&\mapsto& xy& &x&\mapsto& yz& &x&\mapsto& xz\\
y&\mapsto& yz&, &y&\mapsto& xz& ,\ \mathrm{and}&y&\mapsto& xy\\
z&\mapsto& xz& &z&\mapsto& yx& &z&\mapsto& yz
\end{array}.$$
One can then use the definition of compatibility to see that we cannot construct a compatible chain map for any of the three maps above.

So we get that the only cellular maps compatible with the cellular resolution morphisms are the multiplications by a variable.
\end{proof}
Let $t_1$, $t_2$ and $t_3$ be the maps corresponding  to the multiplications by $x$, $y$ and $z$, respectively. 

We note that the possible cellular maps are in one to one correspondence to the morphisms of the minimal family. Figure \ref{minIpic1} shows the corresponding cell complexes of the minimal family to the cell complexes of the non-minimal family of Figure \ref{nonminpic1}.

Next we show that for $n\geq 2$ the cellular maps $t_1$, $t_2$ and $t_3$ are "surjective" together, that is their image cover the cell complex they map to. The condition on $n$ is easily seen by considering the maps from $X_1$ to $X_2$ where no map maps to the central triangle of $X_2$. 

\begin{proposition}
\label{fgtriangles}
For $n\geq 2$, any cell in $X_{n+1}$ is in the image of $t_i(X_n)$, $i\in\{1,2,3\}$, for at least one $i$. 
\end{proposition}
\begin{proof}
Let us consider the three maps between $X_n$ and $X_{n+1}$. 
The maps $t_i$, $i=1,2,3$, are embeddings of $X_n$ into $X_{n+1}$. One can think of these maps as covering a part of $X_{n+1}$ by $X_n$. If three copies of $X_n$ based on $t_i$ cover the whole $X_{n+1}$ we have the desired result. 

The map $t_1$ maps $X_n$ to the subtriangle of $X_{n+1}$ bounded by the vertices with labels containing the variable $x$. It leaves a strip of triangles uncovered. This strip is shown in Figure \ref{trianglestrip}. 
Next consider the map $t_2$. We only need to investigate if it covers any of the cells in the strip, as the other cells are already in the image of $t_1$. The map $t_2$ maps to all labels containing $y$, so in particular it covers all of the triangle strip but the top two ones. This is shown in Figure \ref{trianglestrip}.
Finally we see that the map $t_3$ will cover the remaining cells since it will map to all cells with label $z$ in them, in particular the two top ones.
\end{proof}
\begin{figure}
\begin{tikzpicture}
\filldraw[color=white, fill=blue!5,  thick](1,0) -- (0,2) -- (1,4) -- (2,4) --(3.5,1)--(3,0) -- cycle;
\filldraw[color=red!5, fill=red!5,  thick](4,2)-- (3,4) -- (2,4) --(3.5,1) -- cycle;

\filldraw[color=black, fill=blue!5,  thick](-1,0) -- (0,0) -- (-0.5,1) -- cycle;
\filldraw[color=black, fill=blue!5,  thick](0,0) -- (0.5,1) -- (1,0) -- cycle;
\filldraw[color=black, fill=blue!5,  thick](0,0) -- (0.5,1) -- (0,2) --  (-0.5,1) -- cycle;
\draw (0.5,1) -- (-0.5,1);
\filldraw [black] (-1,0) circle (2pt) node[anchor=east] {$x^{n+1}$};
\filldraw [black] (1,0) circle (2pt) node[anchor=north] {$x^{n-1}y^2$};
\filldraw [black] (0,0) circle (2pt) node[anchor=north] {$x^ny$};
\filldraw [black] (-0.5,1) circle (2pt) node[anchor=east] {$x^nz$};
\filldraw [black] (0,2) circle (2pt) node[anchor=east] {$x^{n-1}z^2$};  
\draw(0,2) --(0.25,2.5);
\draw(0.5,1) --(0.75,1.5);
\draw(1,0) --(1.25,0.5);
\draw(0,2) --(0.5,2);
\draw(0.5,1) --(1,1);
\draw(1,0) --(1.5,0);

\filldraw[color=black, fill=blue!5,  thick](3,0) -- (4,0) -- (3.5,1) -- cycle;
\filldraw[color=black, fill=red!5,  thick](4,0) -- (4.5,1) -- (5,0) -- cycle;
\filldraw[color=black, fill=red!5,  thick](4,0) -- (4.5,1) -- (4,2) --  (3.5,1) -- cycle;
\draw (4.5,1) -- (3.5,1);
\filldraw [black] (3,0) circle (2pt) node[anchor=north] {$y^{n-1}x^2$};
\filldraw [black] (5,0) circle (2pt) node[anchor=west] {$y^{n+1}$};
\filldraw [black] (4,0) circle (2pt) node[anchor=north] {$xy^n$};
\filldraw [black] (4.5,1) circle (2pt) node[anchor=west] {$y^nz$};
\filldraw [black] (4,2) circle (2pt) node[anchor=west] {$y^{n-1}z^2$};  
\draw(3,0)--(2.5,0);
\draw (3.5,1)--(3,1);
\draw (4,2)--(3.5,2);
\draw(3,0)--(2.75,0.5);
\draw (3.5,1)--(3.25,1.5);
\draw (4,2)--(3.75,2.5);

\filldraw[color=black, fill=blue!5,  thick](1,4) -- (2,4) -- (1.5,5) -- cycle;
\filldraw[color=black, fill=red!5,  thick](2,4) -- (2.5,5) -- (3,4) -- cycle;
\filldraw[color=black, fill=red!5,  thick](2,4) -- (2.5,5) -- (2,6) --  (1.5,5) -- cycle;
\draw (2.5,5) -- (1.5,5);
\filldraw [black] (1,4) circle (2pt) node[anchor=east] {$x^2z^{n-1}$};
\filldraw [black] (3,4) circle (2pt) node[anchor=west] {$y^2z^{n-1}$};
\filldraw [black] (1.5,5) circle (2pt) node[anchor=east] {$xz^n$};
\filldraw [black] (2.5,5) circle (2pt) node[anchor=west] {$yz^n$};
\filldraw [black] (2,6) circle (2pt) node[anchor=south] {$z^{n+1}$}; 
\draw(1,4)--(0.75,3.5);
\draw(2,4)--(1.75,3.5);
\draw(3,4)--(2.75,3.5); 
\draw(1,4)--(1.25,3.5);
\draw(2,4)--(2.25,3.5);
\draw(3,4)--(3.25,3.5); 

\draw[dotted](0.25,2.5)--(0.75,3.5);
\draw[dotted](0.75,1.5)--(1.75,3.5);
\draw[dotted](1.25,0.5)--(2.75,3.5);
\draw[dotted](0.5,2)--(3.5,2);
\draw[dotted](1,1)--(3,1);
\draw[dotted](1.5,0)--(2.5,0);
\draw[dotted](2.75,0.5)--(1.25,3.5);
\draw[dotted](3.25,1.5)--(2.25,3.5);
\draw[dotted](3.75,2.5)--(3.25,3.5); 

\filldraw[color=white, fill=blue!5,  thick](10,0) -- (9,2) -- (10,4) -- (11,4) --(12.5,1)--(12,0) -- cycle;
\filldraw[color=blue!5, fill=blue!5,  thick](13,2)-- (12,4) -- (11,4) --(12.5,1) -- cycle;
\filldraw[color=blue!5, fill=blue!10,  thick](11,4)--(12.5,1)--(12,0)--(10,0)--(9.5,1)--cycle;

\filldraw[color=black, fill=blue!5,  thick](8,0) -- (9,0) -- (8.5,1) -- cycle;
\filldraw[color=black, fill=blue!10,  thick](9,0) -- (9.5,1) -- (10,0) -- cycle;
\filldraw[color=black, fill=blue!5,  thick](9,0) -- (9.5,1) -- (9,2) --  (8.5,1) -- cycle;
\draw (9.5,1) -- (8.5,1);
\filldraw [black] (8,0) circle (2pt) node[anchor=east] {$x^{n+1}$};
\filldraw [black] (10,0) circle (2pt) node[anchor=north] {$x^{n-1}y^2$};
\filldraw [black] (9,0) circle (2pt) node[anchor=north] {$x^ny$};
\filldraw [black] (8.5,1) circle (2pt) node[anchor=east] {$x^nz$};
\filldraw [black] (9,2) circle (2pt) node[anchor=east] {$x^{n-1}z^2$};  
\draw(9,2) --(9.25,2.5);
\draw(9.5,1) --(9.75,1.5);
\draw(10,0) --(10.25,0.5);
\draw(9,2) --(9.5,2);
\draw(9.5,1) --(10,1);
\draw(10,0) --(10.5,0);

\filldraw[color=black, fill=blue!10,  thick](12,0) -- (13,0) -- (12.5,1) -- cycle;
\filldraw[color=black, fill=blue!5,  thick](13,0) -- (13.5,1) -- (14,0) -- cycle;
\filldraw[color=black, fill=blue!5,  thick](13,0) -- (13.5,1) -- (13,2) --  (12.5,1) -- cycle;
\draw (13.5,1) -- (12.5,1);
\filldraw [black] (12,0) circle (2pt) node[anchor=north] {$y^{n-1}x^2$};
\filldraw [black] (14,0) circle (2pt) node[anchor=west] {$y^{n+1}$};
\filldraw [black] (13,0) circle (2pt) node[anchor=north] {$xy^n$};
\filldraw [black] (13.5,1) circle (2pt) node[anchor=west] {$y^nz$};
\filldraw [black] (13,2) circle (2pt) node[anchor=west] {$y^{n-1}z^2$};  
\draw(12,0)--(11.5,0);
\draw (12.5,1)--(12,1);
\draw (13,2)--(12.5,2);
\draw(12,0)--(11.75,0.5);
\draw (12.5,1)--(12.25,1.5);
\draw (13,2)--(12.75,2.5);

\filldraw[color=black, fill=blue!5,  thick](10,4) -- (11,4) -- (10.5,5) -- cycle;
\filldraw[color=black, fill=blue!5,  thick](11,4) -- (11.5,5) -- (12,4) -- cycle;
\filldraw[color=black, fill=red!5,  thick](11,4) -- (11.5,5) -- (11,6) --  (10.5,5) -- cycle;
\draw (11.5,5) -- (10.5,5);
\filldraw [black] (10,4) circle (2pt) node[anchor=east] {$x^2z^{n-1}$};
\filldraw [black] (12,4) circle (2pt) node[anchor=west] {$y^2z^{n-1}$};
\filldraw [black] (10.5,5) circle (2pt) node[anchor=east] {$xz^n$};
\filldraw [black] (11.5,5) circle (2pt) node[anchor=west] {$yz^n$};
\filldraw [black] (11,6) circle (2pt) node[anchor=south] {$z^{n+1}$}; 
\draw(10,4)--(9.75,3.5);
\draw(11,4)--(10.75,3.5);
\draw(12,4)--(11.75,3.5); 
\draw(10,4)--(10.25,3.5);
\draw(11,4)--(11.25,3.5);
\draw(12,4)--(12.25,3.5); 

\draw[dotted](9.25,2.5)--(9.75,3.5);
\draw[dotted](9.75,1.5)--(10.75,3.5);
\draw[dotted](10.25,0.5)--(11.75,3.5);
\draw[dotted](9.5,2)--(12.5,2);
\draw[dotted](10,1)--(12,1);
\draw[dotted](10.5,0)--(11.5,0);
\draw[dotted](11.75,0.5)--(10.25,3.5);
\draw[dotted](12.25,1.5)--(11.25,3.5);
\draw[dotted](12.75,2.5)--(12.25,3.5); 
\end{tikzpicture}
\caption{(a) The image of $t_1$ is in blue and the uncover triangle strip is in light red in $X_{n+1}$. (b) The images of $t_1$ and $t_2$ in blue and uncovered top triangles.}
\label{trianglestrip}
\end{figure}
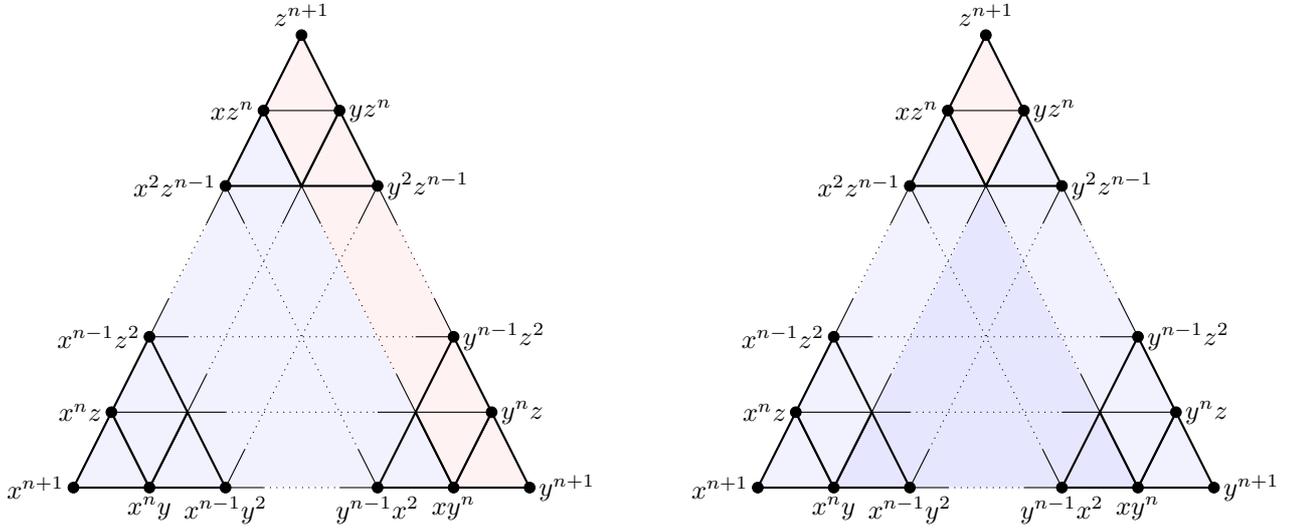
The overlap of any two of the maps is $X_{n-1}$ with labels multiplied with $xy$, $xz$ or $yz$ with respect to the maps. The overlap of all three is $X_{n-2}$ with labels multiplied by $xyz$. Here we consider $X_0$ to be a single point with label 1.

\begin{figure}
\begin{tikzpicture}
\filldraw[color=black, fill=red!5,  thick](-3,0) -- (-2.5,1) -- (-2,0) -- cycle;
\filldraw [black] (-3,0) circle (2pt) node[anchor=east] {$x$};
\filldraw [black] (-2,0) circle (2pt) node[anchor=west] {$y$};
\filldraw [black] (-2.5,1) circle (2pt) node[anchor=west] {$z$};

\filldraw[color=black, fill=red!5,  thick](-1,0) -- (0,0) -- (-0.5,1) -- cycle;
\filldraw[color=black, fill=red!5,  thick](0,0) -- (0.5,1) -- (1,0) -- cycle;
\filldraw[color=black, fill=red!5,  thick](0,0) -- (0.5,1) -- (0,2) --  (-0.5,1) -- cycle;
\filldraw [black] (-1,0) circle (2pt) node[anchor=east] {$x^2$};
\filldraw [black] (1,0) circle (2pt) node[anchor=west] {$y^2$};
\filldraw [black] (0,0) circle (2pt) node[anchor=north] {$xy$};
\filldraw [black] (-0.5,1) circle (2pt) node[anchor=east] {$xz$};
\filldraw [black] (0.5,1) circle (2pt) node[anchor=west] {$yz$};
\filldraw [black] (0,2) circle (2pt) node[anchor=south] {$z^2$};

\filldraw[color=black, fill=red!5,  thick](2,0) -- (3,0) -- (2.5,1) -- cycle;
\filldraw[color=black, fill=red!5,  thick](3,0) -- (3.5,1) -- (4,0) -- cycle;
\filldraw[color=black, fill=red!5,  thick](3,0) -- (3.5,1) -- (3,2) --  (2.5,1) -- cycle;
\filldraw[color=black, fill=red!5,  thick](4,0) -- (4.5,1) -- (5,0) -- cycle;
\filldraw[color=black, fill=red!5,  thick](4,0) -- (4.5,1) -- (4,2) --  (3.5,1) -- cycle;
\filldraw[color=black, fill=red!5,  thick](3.5,1) -- (4,2) -- (3.5,3) --  (3,2) -- cycle;
\filldraw [black] (2,0) circle (2pt) node[anchor=east] {$x^3$};
\filldraw [black] (4,0) circle (2pt) node[anchor=north] {$xy^2$};
\filldraw [black] (3,0) circle (2pt) node[anchor=north] {$x^2y$};
\filldraw [black] (2.5,1) circle (2pt) node[anchor=east] {$x^2z$};
\filldraw [black] (3.5,1) circle (2pt) node[anchor=west] {$xyz$};
\filldraw [black] (3,2) circle (2pt) node[anchor=east] {$xz^2$};
\filldraw [black] (3.5,3) circle (2pt) node[anchor=west] {$z^3$};
\filldraw [black] (4,2) circle (2pt) node[anchor=west] {$yz^2$};
\filldraw [black] (4.5,1) circle (2pt) node[anchor=west] {$y^2z$};
\filldraw [black] (5,0) circle (2pt) node[anchor=west] {$y^3$};

\filldraw[color=black, fill=red!5,  thick](6,0) -- (7,0) -- (6.5,1) -- cycle;
\filldraw[color=black, fill=red!5,  thick](7,0) -- (7.5,1) -- (8,0) -- cycle;
\filldraw[color=black, fill=red!5,  thick](7,0) -- (7.5,1) -- (7,2) --  (6.5,1) -- cycle;
\filldraw[color=black, fill=red!5,  thick](8,0) -- (8.5,1) -- (9,0) -- cycle;
\filldraw[color=black, fill=red!5,  thick](8,0) -- (8.5,1) -- (8,2) --  (7.5,1) -- cycle;
\filldraw[color=black, fill=red!5,  thick](7.5,1) -- (8,2) -- (7.5,3) --  (7,2) -- cycle;
\filldraw[color=black, fill=red!5,  thick](9,0) -- (9.5,1) -- (10,0) -- cycle;
\filldraw[color=black, fill=red!5,  thick](9,0) -- (9.5,1) -- (9,2) --  (8.5,1) -- cycle;
\filldraw[color=black, fill=red!5,  thick](8.5,1) -- (9,2) -- (8.5,3) --  (8,2) -- cycle;
\filldraw[color=black, fill=red!5,  thick](8,2) -- (8.5,3) -- (8,4) --  (7.5,3) -- cycle;
\filldraw [black] (6,0) circle (2pt) node[anchor=east] {$x^4$};
\filldraw [black] (8,0) circle (2pt) node[anchor=north] {$x^2y^2$};
\filldraw [black] (7,0) circle (2pt) node[anchor=north] {$x^3y$};
\filldraw [black] (6.5,1) circle (2pt) node[anchor=east] {$x^3z$};
\filldraw [black] (7.5,1) circle (2pt) node[anchor=north] {$x^2yz$};
\filldraw [black] (7,2) circle (2pt) node[anchor=east] {$xz^2$};
\filldraw [black] (7.5,3) circle (2pt) node[anchor=east] {$xz^3$};
\filldraw [black] (8,2) circle (2pt) node[anchor=west] {$xyz^2$};
\filldraw [black] (8.5,1) circle (2pt) node[anchor=west] {$xy^2z$};
\filldraw [black] (9,0) circle (2pt) node[anchor=north] {$xy^3$};
\filldraw [black] (8,4) circle (2pt) node[anchor=west] {$z^4$};
\filldraw [black] (8.5,3) circle (2pt) node[anchor=west] {$yz^3$};
\filldraw [black] (9,2) circle (2pt) node[anchor=west] {$y^2z^2$};
\filldraw [black] (9.5,1) circle (2pt) node[anchor=west] {$y^3z$};
\filldraw [black] (10,0) circle (2pt) node[anchor=west] {$y^4$};

\end{tikzpicture}
\caption{A minimal family of cell complexes supporting the resolution of powers of $I$.}
\label{minIpic1}
\end{figure}
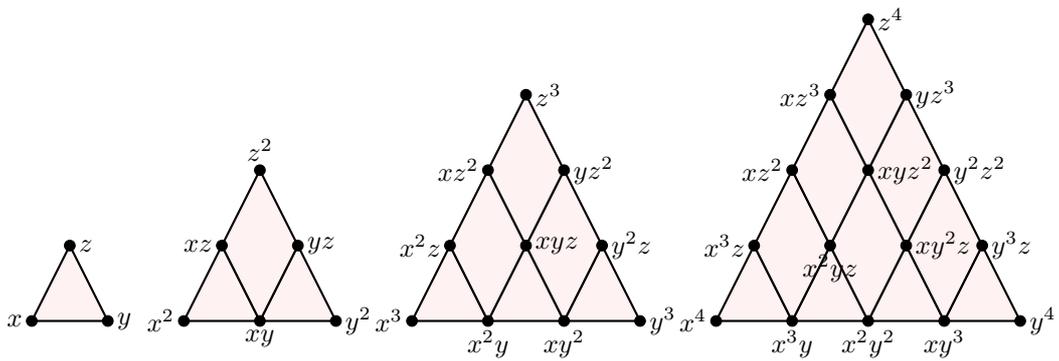
Let $\F$ be the family given by the cellular resolutions $S/I^n$ supported on the subdivided triangles. We also have the family $\overline{\F}$ that consists of the minimal resolutions of $S/I^n$, which is also a subcategory of {\bf CellRes}. The morphisms are pairs of  compatible maps, such that we only have one pair for each chain map. The first few cell complexes of the resolutions are shown in Figure \ref{minIpic1}.

We want to show that the category of representations is noetherian. One way for this is to study the principal projectives.
In $\F$ the Hom sets are finite, so the principal projective $P_x$ gives a finitely generated free module for any element. Moreover since any map $F_n\rightarrow F_m$ comes from the composition of maps of the form $F_i\rightarrow F_{i+1}$ for $n\leq i\leq m-1$, we get that the principal projectives are finitely generated by definition. 

\begin{proposition}
\label{ppnoeth}
Every principal projective representation $P_x$ of $\overline{\F}$ is noetherian.
\end{proposition}
\begin{proof}
By definition a representation is noetherian if every ascending chain of subobjects stabilises. In $\overline{\F}$ all the ascending chains are the whole category or some subset of it.  Since the principal projective is finitely generated the same arguments show the ascending chains stabilise in the representation stability sense.
\end{proof}

\begin{corollary}
$\textrm{Rep}_S(\overline{\F})$ is noetherian.
\end{corollary}
\begin{proof}
 By Proposition \ref{ppnoeth} every principal projective is noetherian. Then applying Proposition \ref{noeth} to $\overline{\F}$ gives the noetherianity of the category.
\end{proof}

Let us consider the following representation of $\overline{\F}$. We define a functor
$$s_p:\overline{\F}\rightarrow \smod$$
taking a cellular resolution of $I^n$ to the $p^{th}$ module of the non-minimal resolution of $I^n$. The functor maps a morphism given by a multiplication of a monomial to the $p^{th}$ component of chain map between the non-minimal resolutions which is the multiplication by the same monomial between the non-minimal resolutions. 
\begin{proposition}
The representation $s_p$ of $\overline{\F}$ defined above is finitely generated.
\end{proposition}
\begin{proof}
Since the representation $s_p$ takes a morphism of $\overline{\F}$ to the map between the modules given by a corresponding chain map, we have that the morphisms of $\overline{\F}$ give all the possible maps between the modules that come from a compatible chain map. Therefore 
 Proposition~\ref{fgtriangles} shows us that the  representation $s_p$ is finitely generated by definition.
\end{proof}

Then we get the following.
\begin{proposition}
The representation 
$$\sigma_p:\overline{\F}\rightarrow \smod$$ taking resolution $\mathcal{F}$ to its $p^{th}$ syzygy module is finitely generated.
\end{proposition}
\begin{proof}
First we note that the $p^{th}$ syzygy module is contained in the $p^{th}$ module of the non-minimal resolution. So we get that $\sigma_p$ is a subrepresentation of $s_p$. Further more we know that in a noetherian representation category any subrepresentation of finitely generated representation is finitely generated. Thus $\sigma_p$ is finitely generated as $s_p$ is.
\end{proof}

\subsection{Noetherianity results for families of cellular resolutions}
We begin this section by stating a proposition about the noetherianity properties of linear families. 
\begin{proposition}
\label{noethrep}
Let $\F$ be a linear family of cellular resolutions with finitely generated Hom sets. Then $\mathrm{Rep}_S(\F)$ is noetherian.
\end{proposition}
\begin{proof}
We prove the noetherianity by showing that every ascending chain stabilises under all principal projectives.
First fix a principal projective $P_x$. 
 
By definition of a subobject, the only subobject for some $F_i$ in $\F$ are the cellular resolutions $F_k$ with $k<i$. Then it follows that the possible ascending chains of subobjects are either the whole family, or some subset of it.  

Next we want to study the behaviour of the principal projective $P_x$ on the different ascending chains of subobjects. Denote the cellular resolution corresponding to $x$ by $F_i$ for some $i$. Then we have no maps to resolutions $F_k$ for $k<i$, and so $P_x(F_k)=S[\mathrm{Hom}(F_i,F_k)]=S$ for $k<i$. Thus for the stability point of view it suffices to look at the family $\F$ from $F_i$ onwards. So we have 
$$ \F_{\geq i}: F_i\rightarrow F_{i+1}\rightarrow \ldots \rightarrow F_j\rightarrow\ldots$$
Applying the principal projective functor to the family $\F_{\geq i}$ we get the following sequence of free modules
$$P_x(\F_{\geq i}): S[\mathrm{Hom}(F_i,F_i)]\rightarrow S[\mathrm{Hom}(F_i, F_{i+1})]\rightarrow \ldots\rightarrow S[\mathrm{Hom}(F_i,F_j)]\rightarrow\ldots$$
The maps between $S[\mathrm{Hom}(F_i,F_j)]\rightarrow S[\mathrm{Hom}(F_i, F_{j+1})]$ are given by sending the generator $e_f\in S[\mathrm{Hom}(F_i,F_j)]$, $f\in \mathrm{Hom}(F_i,F_j)$, to $e_{g\circ f}\in S[\mathrm{Hom}(F_i, F_{j+1})]$ where $g\in \mathrm{Hom}(F_j,F_{j+1})$. Each morphism in $ \mathrm{Hom}(F_j,F_{j+1})$ gives a map between the free modules. For rest of the proof we refer to these maps given by the post-composition in Hom-sets as post composition by a morphism, denoted by $p_g$ for $g\in \mathrm{Hom}(F_j,F_{j+1})$. 

If we have a map from $h:S[\mathrm{Hom}(F_i,F_j)]\rightarrow S[\mathrm{Hom}(F_i,F_k)]$ for some $k>j>i$, then using the requirement for the family that all morphisms are made of compositions between the consecutive resolutions, we can consider each of the components of $h$. First we can write the map as 
$$S[\mathrm{Hom}(F_i,F_j)]\xrightarrow{h'} S[\mathrm{Hom}(F_i,F_{k-1})]\xrightarrow{p_g} S[\mathrm{Hom}(F_i,F_k)],$$
where $p_g$ is a suitable post-composition. This process can repeated and choosing the suitable compositions between each consecutive modules until we get $h$ as a composition of $p_g$s. Note that this is not necessarily unique decomposition.
Another important thing to note about the maps $p_g$ is that any generator in $S[\mathrm{Hom}(F_i, F_{j+1})]$ is given by the form  $p_g(e_f)=e_{g\circ f}$ where $e_f\in S[\mathrm{Hom}(F_i,F_j)]$, due to linear structure of the family.

Therefore if we have the full family, any generator of $S[\mathrm{Hom}(F_i,F_j)]$ is given by applying $j-i$ $p_g$ maps to the generators of $S[\mathrm{Hom}(F_i,F_i)]$. Hence in this case the principal projective is finitely generated or in other words it stabilises. 

If we are looking at some subfamily as the ascending chain, with the indices of the cellular resolutions denoted by $j_1,j_2,j_3,\ldots$, then the sequence looks the following after applying the principal projective
$$S[\mathrm{Hom}(F_i,F_{j_1})]\rightarrow S[\mathrm{Hom}(F_i, F_{j_2})]\rightarrow \ldots\rightarrow S[\mathrm{Hom}(F_i,F_{j_k})]\rightarrow\ldots$$
Again any map $S[\mathrm{Hom}(F_i,F_{j_k})]\rightarrow S[\mathrm{Hom}(F_i, F_{j_l})]$ becomes a composition of $p_g$ maps. Note that in this case we might have two  $p_g\circ p_{g'}$ or more as map between two free modules but not the individual components. Since the morphisms have not been restricted all maps from $\mathrm{Hom}(F_i,F_{j_k})$ to $\mathrm{Hom}(F_i, F_{j_l})$ are given by $k-l$ post-compositions, and any map in $\mathrm{Hom}(F_i, F_{j_l})$ consist of a map from $\mathrm{Hom}(F_i,F_{j_k})$ and post-compositions of consecutive maps. Thus we have that any generator in $S[\mathrm{Hom}(F_i,F_{j_k})]$ can written as image of some generator in $S[\mathrm{Hom}(F_i,F_{j_1})]$ which is finitely generated free module, and so we get that any subsequence in the family is finitely generated under the principal projective $P_x$. 

We have shown that any ascending chain of subobjects will be finitely generated under an arbitary principal projective, equivalently all ascending chains of subobjects stabilise under the principal projectives. Therefore all the principal projectives are noetherian and by Proposition \ref{noeth} $\mathrm{Rep}_S(\F)$ is noetherian.
\end{proof}


\begin{remark}
The noetherian representation depends on the morphisms, so for a general family of cellular resolutions we need at least two conditions to have noetherian representation category:
\begin{itemize}
\item finitely generated morphisms, i.e. in most cases quotient out the morphisms with the same chain map.
\item After finitely many steps the morphism can all be described by just giving the morphism between two consecutive cellular resolutions.
\end{itemize}
\end{remark}
The above remark gives us the following corollary of Proposition \ref{noethrep}.
\begin{corollary}
If $\F$ is a family of cellular resolutions such that it is linear after some finite sequence of length $i$ and for all the resolutions $F_j, F_k$,$j,k\leq i$, $\mathrm{Hom}(F_j,F_k)$ is finitely generated, then $\mathrm{Rep}_S(\F)$ is noetherian. 
\end{corollary}
\begin{proof}
If we look at the family consisting of the part $\F_{>i}$, then by Proposition \ref{noethrep} any principal projective is finitely generated on any subsequence. Since the discarded part in the sequence is finite, we only have finitely many Hom-sets between the cellular resolutions, and each of these sets is finitely generated. So when considering the whole family $\F$ or a subsequence that contains cellular resolutions from the first $i$ resolutions, we only need to add finitely many finite generating sets to the generators of $\F_{>i}$ under any principal projective $P_x$. Thus it is still finitely generated, and we have noetherianity for every principal projective and by Proposition \ref{noeth} also for $\mathrm{Rep}_S(\F)$.
\end{proof}

\subsection{Gr\"obner families of cellular resolutions}
\label{grobnerfams}
A common example of a linear family of cellular resolutions is the family consisting of powers o an ideal with some assumptions on the morphisms. Then one naturally wonders whether these very controlled families also satisfy the conditions of being Gr\"obner. For this we consider the following special type of a family.

Let $I$ be a monomial ideal with $m$ generators $g_1,g_2,\ldots,g_m$. Suppose that for each power of $I$, the module $S/I^k$ has a cellular resolution.
Let 
$$\F: F_1\rightarrow F_2\rightarrow \ldots\rightarrow F_i\rightarrow\ldots$$
be a family of cellular resolutions where $F_i=S/I^i$. Furthermore, let the only maps between consecutive resolutions in the family be multiplications by the generators of $I$ (again one morphism for each chain map) and the only map from $F_i$ to itself is the identity. Moreover, let us suppose all other maps are compositions of the multiplications. This kind of family is not only a small category but also directed and so we can use the Proposition \ref{bettergrobner} to study if it is  Gr\"obner.

\begin{example}
Let us consider the ideal $I=(x,y,z)$ and the cellular resolutions of the powers again and let $\F$ denote the family of non-minimal triangle resolutions. This family satisfies the conditions described above and in particular it is a directed category. 
Let $\F_F$ denote the category of morphisms from $F\in \F$ where the morphisms are commuting triangles. 
We want to study the set $|\F_F|$ for some $F\in\F$. 

First we want to show that the set $|\F_F|$ has an admissible order. We label each $f\in|\F_F|$ by the monomial multiplication it is associated to, and then take an ordering on $|\F_F|$ given by any monomial order on the monomials. This then gives us the admissible order on $|\F_F|$.

The set $|\F_F|$ is a poset with the order given by $x\leq y$ if there is a map $x\rightarrow y$. Moreover it is noetherian poset. This can be seen by considering the requirement for descending chain condition and no anti-chains. Take a descending chain 
$$f_1\geq f_2\geq f_3\geq\ldots$$
By definition of the order this means we must have a chain of maps 
$$F_1\leftarrow F_2\leftarrow F_3\leftarrow\ldots$$
in our category where the maps go in the direction of increasing powers. Hence the chain of decreasing cellular resolutions cannot go on forever but will either have to reach the resolution of $S/I$ or stabilise before that. Next we want to look at anti-chains. Observe that if we have $f_i: F\rightarrow F_i$ and $f_j:F\rightarrow F_j$ in $|\F_F|$ then we will have either $F_i\rightarrow F_j$ or $F_j\rightarrow F_i$ depending on which is the resolution of higher power. Thus we get that $f_i\geq f_j$ or $f_j\geq f_i$ for $i\neq j$. 
Then if we want an anti-chain we must have that all elements in it correspond to maps to the same cellular resolution. However, we only have finitely many such maps and so we can only have finite anti-chains. 

Then by Proposition \ref{bettergrobner} the category $\F$ is Gr\"obner. 
\end{example}

With the above example in mind we formulate the following proposition.
\begin{proposition}
Let $I$ be a monomial ideal.
Let $\F$ be a family of cellular resolutions where $F_i=S/I^i$ and let the only maps between consecutive resolutions in the family be multiplications by the generators of $I$ and the only map from $F_i$ to itself is the identity.
Then $\F$ is Gr\"obner.
\end{proposition}
\begin{proof}

Let $\F$ be a family of cellular resolutions as specified in the proposition and let $I$ be the defining ideal. Then since all the $\operatorname{Hom}$-sets are finite, and since we have no  selfmaps other than the identity, the family $\F$ is small and directed as a category. Therefore it suffices to study the set $|\F_F|$ for some arbitrary member $F$ of the family.  Recall that $\F_F$ is the category of all arrows from $F$ with commuting triangles as morphisms, and that $|\F_F|$ is the set of isomorphism classes. 

Each morphisms in $\F$ corresponds to multiplication by some monomial that consists of multiples of the monomial generators of $I$. If we consider the arrows in the category $\F_F$, we note that there are no objects that are arrows to powers smaller than $F$, so we only have objects corresponding to maps to powers higher than $F$ has. In the set of  isomorphism classes $|\F_F|$, we note that if two morphisms are associated to different multiplication they are not isomorphic and if we have two morphisms that correspond to the same monomial multiplication, then these are isomorphic. Therefore we can label our morphisms uniquely with the monomials corresponding to multiplication. Then taking any monomial order, say for example lexicographic, will give an admissible order on the set $|\F_F|$.

Next we want to look at $|\F_F|$ as a poset with the natural order $f\leq g$ if there exists a morphism $f\rightarrow g$. We want to show that this poset is noetherian. First let us consider any anti-chain in it, that is a chain of elements such that any two are not comparable. Let $f:F\rightarrow G$ and $g:F\rightarrow G'$ be two noncomparable objects. This means there is no map between $G$ and $G'$ that forms a commutative diagram, or in the other direction. By defintion the family satisfies that between any consecutive objects we have multiplication by the monomials in $I$, and all other maps are compositions of these. Then if $G$ and $G'$ are not the same resolutions, we can find a map between them that gives a commutative triangle with $f$ and $g$. This tells us that any anti-chain must be made of arrows to the same resolution. We only have finitely many of such objects, hence there cannot be an infite anti-chain.

Next we want to look at the descending chain condition on $|\F_F|$. Take any descending chain $f_1\geq f_2\geq f_3\geq...$ and consider the maps that form it as shown in Figure \ref{gmaps}.

\begin{figure}
\begin{center}
\begin{tikzpicture}
\node (A) at (-2,5) {$F$};
\node (1) at (3,5) {$G_1$};
\node (2) at (3,3.5) {$G_2$};
\node(3) at (3,2) {$G_3$};
\node (4) at (3,0.5) {$G_4$};
\node(5) at (3,-1) {$\vdots$};
\node(6) at (0.5,1.9) {$\vdots$};
\node(B) at (0.5,5.1) {$f_1$};
\node(C) at (0.5,4.1) {$f_2$};
\node(D) at (0.5,3.4) {$f_3$};
\node(E) at (0.5,2.5) {$f_4$};

\draw[->] (A)--(1);\draw[->]  (A)--(2);\draw[->]  (A)--(3);\draw[->]  (A)--(4); \draw[->] (5)--(4);\draw[->] (2)--(1);\draw[->] (3)--(2);\draw[->] (4)--(3);
\end{tikzpicture}
\end{center}
\caption{Commuting triangles in $|\F_F|$.}
\label{gmaps}
\end{figure}
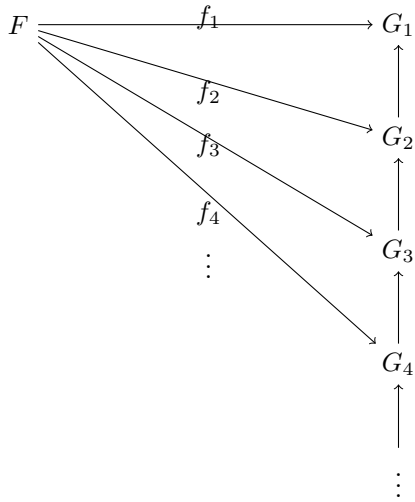
The vertical chain of maps is what gives the order relation. From the definition of the family we get that the chain of the vertical maps cannot go on forever as eventually we reach the resolution of $S/I$ and there are no resolutions mapping to it or our chain will stabilise before it. In either case eventually we must have $f_i=f_i+1$ in the descending chain as we only have one map that can be repeated. Note that we cannot end in the situation where the resolution that $F$ maps to are the same but we rotate possible maps as these are precisely the non-comparable cases.  Thus the poset $|\F_F|$ satisfying the descending chain condition and has no infinite anti-chains, so it is noetherian.

Since our choice of $F$ was arbitrary, we can then applying Proposition \ref{bettergrobner} gives us that $\F$ is a Gr\"obner category.
\end{proof}

\section{The syzygy functor}
The main interest for us in the families of cellular resolutions is on their syzygies. The $p$-th syzygy module can be written as a representation of a family $\F$. First we define a representation on the $p$-th free module.
\begin{definition}
The p-th module representation is a functor
$$s_p: \F\rightarrow \rmod$$
 such that $s_p(F_i)=p$-th free module in the resolution and $s_p(F_i\rightarrow F_j)$ is the restriction of the chain map from $F_i$ to $F_j$ on the $p$-th component.
\end{definition}

Next we define the syzygy representation for a family of cellular resolutions. 
\begin{definition}
\label{syzf}
Let $\F$ be a family of cellular resolutions. The $p$-th syzygy functor 
$$\sigma_p:\F\rightarrow \rmod$$
is defined by taking $F\in \F$ to its $p$-th syzygy module and the morphisms are restrictions of the chain maps. 
\end{definition}

\begin{proposition}
\label{subrep}
The representation $\sigma_p$ is a subrepresentation of $s_p$.
\end{proposition}
\begin{proof}
We know that the minimal resolution is contained in a non-minimal one as a direct summand. 
If the family consist of minimal cellular resolutions, then  $\sigma_p$ is a finitely generated submodule of $s_p$ by definition of the syzygy module.
Otherwise the family $\F$ has at least one non-minimal resolution. Then we can take a natural transformation $\eta$ from $\sigma_p$ to $s_p$ given by an embedding. Then the natural transformation will have the maps $\eta_F:\sigma_p(F)\rightarrow s_p(F)$ to be monic and hence $\sigma_p$ is a subfunctor of $s_p$.
\end{proof}

We are interested in studying the finite generation of the syzygy functor. To do this we want to make use of the cell complex structure in cellular resolutions and this gives rise to the following definition of covering.
\begin{definition}
\label{coveringdef}
Let $F$ and $G$ be cellular resolutions such that $\mathrm{Hom}(F,G)$ is not empty. Let $X$ be the cell complex supporting $F$ and let $Y$ be the cell complex supporting $G$. We say that we have \emph{covering} of $Y$ by $X$ if the images of $X$ under the maps $f\in \mathrm{Hom}(F,G)$ cover $Y$, $\cup_{f\in \mathrm{Hom}(F,G)}f(X)=Y$.

Let $F_1,F_2,\ldots,F_r$ be cellular resolutions mapping to $G$. Let $X_i$ be the cell complex supporting $F_i$ and $Y$ be the cell complex supporting $G$. Then we say that \emph{$X$ is a covering of $Y$} if the images of $X_i$ under the maps $f\in \mathrm{Hom}(F_i,G)$ cover $Y$, $\cup_{f\in \mathrm{Hom}(F_i,G)}f(X_i)=Y$.
\end{definition}

\begin{definition}
Let $F_1,F_2,\ldots,F_r$ be cellular resolutions mapping to $G$. Let $X_i$ be the cell complex supporting $F_i$ and $Y$ be the cell complex supporting $G$. Then we say that we have \emph{d-covering} of $Y$ by $X_1,X_2,\ldots,X_r$ if the images of $d$-cells of $X_i$ cover $d$-cells of $Y$ under the maps $f\in \mathrm{Hom}(F_i,G)$ , $\cup_{f\in \mathrm{Hom}(F_i,G)}f(X_i)=Y$.
\end{definition}
\begin{remark}
In general taking the Taylor resolution family will not give a covering. This is due to not having enough maps between the simplices, see Section \ref{dimproblem} for more details.
\end{remark}

\begin{lemma}
\label{dcover}
The p-th module representation is finitely generated if and only if we have  $(p-1)$-covering of  $X_i$ by finitely many $X_j$s with $j<i$ for all $i$ large enough. 
\end{lemma}
\begin{proof}
First let us fix $p$.
Let $\F$ be a family of representations and suppose that the p-th module functor $s_p$ is finitely generated. By definition this means we have finitely many generators $\epsilon_1,\epsilon_2,\ldots,\epsilon_r$ in $s_p(F_i)$,$s_p(F_i),ldots, s_p(F_i)$ such that they generate all the modules $s_p(\F)$. Alternatively any generator $e$ of $F_j$, $j>i$, can be written as an image of one or more of the $\epsilon$s via maps in the sets $\mathrm{Hom}(F_k,F_j)$, $1\leq k\leq i$. 

Each of the generators $\epsilon_1,\epsilon_2,\ldots,\epsilon_r$ corresponds to a cell in the first $i$ cell complexes, and moreover these are all cells of dimension $p-1$. Denote these cells by $c_1,c_2,\ldots,c_r$. From the compatibility of the cellular resolution maps, if there is a generator mapping to a generator then the corresponding cells map to each other. Then the finite generation implies that choosing any $p-1$ dimensional cell $c$ in some $X_j$, there is a cellular map $g$ belonging to a cellular resolution morphism such that $g(c_k)=c$ for some $1\leq k\leq r$. Then any $X_i$ has a $(p-1)$-covering by finitely many $X)j$s, $j<i$, for large enough $i$.

Conversely, assume that we have a covering of $X_i$ by finitely many $X_j$s with $j<i$ for all $i$ large enough. 
So given any $X_j$, $j>i$, there are cell complexes $X_{k_1},X_{k_2},\ldots,X_{k_r}$ that cover $X_j$. We want to show that all $X_{k_1},X_{k_2},\ldots,X_{k_r}$ are below $i$. Suppose that one of them is not, call this one $X_k$. Since $k>i$ we have that its $(p-1)$-covering of it by some $X_{l_1},X_{l_2},\ldots,X_{l_s}$.
Note that composing the cellular map that the $X_{l_1},X_{l_2},\ldots,X_{l_s}$s have to $X_k$ and needed maps from $X_k$ to $X_j$, give a map from $X_{l_1},X_{l_2},\ldots,X_{l_s}$ to $X_j$. Since $X_{l_1},X_{l_2},\ldots,X_{l_s}$ give $(p-1)$-covering of $X_k$ we can replace $X_k$ in the covering set of cell complexes by the $X_{l_1},X_{l_2},\ldots,X_{l_s}$. If all of the $X_{l_1},X_{l_2},\ldots,X_{l_s}$ are below $i$, then we have $(p-1)$-covering of $X_j$ by cell complexes below $i$. Otherwise, we repeat the process for the cell complexes above $i$ to get $(p-1)$-coverings by finitely many cell complexes below $i$.
The process will always give cell complexes below $i$, otherwise we would contradict the composition properties of the morphisms.

Taking the first $i$ cell complexes such that we have $(p-1)$-covering for $X_j$, $j>i$, then their $(p-1)$-cells cover all the other $(p-1)$-cells. Again from the compatibility of the cellular resolution maps, this means any generator of $s_p(F_j)$ is reachable via the maps from one of the generators of the first $i$ modules for large enough $j$. we have finitely many $(p-1)$-cells in the first $i$ cell complexes, hence the first $i$ free modules $s_p(F_i)$ have finitely many generators all together. Thus we get a finite generating set for the representation $s_p$.
\end{proof}


\begin{theorem}
\label{covering}
If $\F$ is a family of cellular resolutions with noetherian representation category $\mathrm{Rep}_S(\F)$ such that the cell complex supporting $F_i$ is covered by the cell complexes supporting $F_{j}$, $j<i$, for all $i$ large enough.
Then the syzygy representation $\sigma_p$ is finitely generated for all $p$.
\end{theorem}
\begin{proof}
Let $F_i$ denote the cellular resolutions in the family and let $X_i$ be the cell complex supporting $F_i$. 

If we have a covering of the cell complexes, then in particular we have a $d$-covering of $F_i$ for all dimensions $d$ by some finite number of lower cellular resolutions for large enough $i$. Then by Lemma \ref{dcover} we have that $s_d$ is finitely generated for all $d$. 

Now Proposition \ref{subrep} tells us that $\sigma_p$ is a subrepresenation of $s_p$ for all $p$. Since $s_p$ is finitely generated, then by noetherianity any subrepresentation is also finitely generated. Hence $\sigma_p$ is finitely generated for all $p$. 
\end{proof}
\begin{remark}
If one fixes a family of ideals, then the modules we get from them may have multiple cellular resolutions, minimal or non-minimal. Then showing that there is no covering for one of the possible cellular resolutions does not imply that the others might not have it.  For example if we look at the Taylor resolution for the ideal $I=(x,y,z)$, we will not get covering of the cell complexes with only three maps. 
\end{remark}

\begin{remark}
The condition on having covering for everything above large enough $i$ is a needed condition. For example consider the following family:
\end{remark}

\begin{example}
\label{squarex}
Let $S=k[x,y,z,w]$ be a graded polynomial ring and let $I=(xz,xw,yz,yw)$ be an ideal. Let us consider the family of cellular resolutions consisting of $S/I^n$. Each of these modules has a minimal cellular resolution, and the cell complexes supporting the first four modules $S/I,S/I^2,S/I^3$ and $S/I^4$ are shown in Figure \ref{squares}.

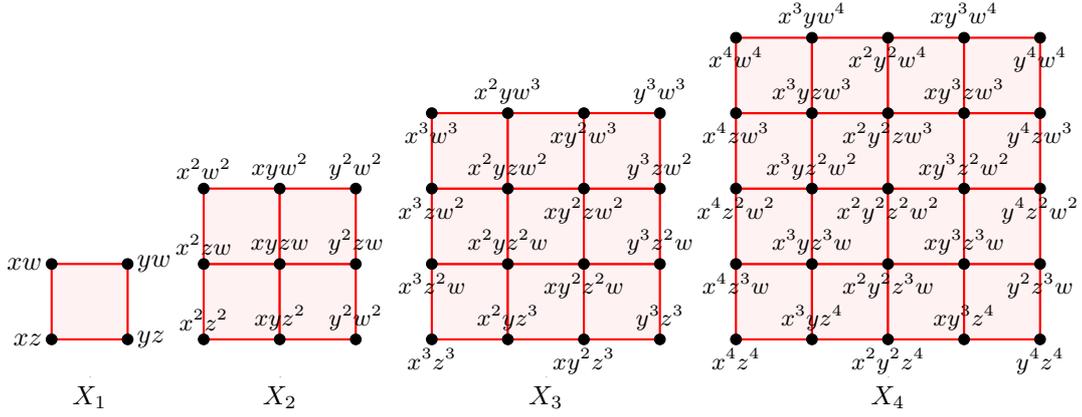
\begin{figure}
\begin{tikzpicture}
\filldraw[color=red, fill=red!5,  thick](0,0) -- (1,0) -- (1,1) -- (0,1) -- cycle;
\filldraw [black] (0,0) circle (2pt) node[anchor=east] {$xz$};
\filldraw [black] (1,0) circle (2pt) node[anchor=west] {$yz$};
\filldraw [black] (1,1) circle (2pt) node[anchor=west] {$yw$};
\filldraw [black] (0,1) circle (2pt) node[anchor=east] {$xw$};

\filldraw[color=red, fill=red!5,  thick](2,0) -- (3,0) -- (3,1) -- (2,1) -- cycle;
\filldraw[color=red, fill=red!5,  thick](3,0) -- (4,0) -- (4,1) -- (3,1) -- cycle;
\filldraw[color=red, fill=red!5,  thick](2,1) -- (3,1) -- (3,2) -- (2,2) -- cycle;
\filldraw[color=red, fill=red!5,  thick](3,1) -- (4,1) -- (4,2) -- (3,2) -- cycle;
\filldraw [black] (2,0) circle (2pt) node[anchor=south] {{\small $x^2z^2$}};
\filldraw [black] (3,0) circle (2pt) node[anchor=south] {\small $xyz^2$};
\filldraw [black] (3,1) circle (2pt) node[anchor=south] {\small $xyzw$};
\filldraw [black] (2,1) circle (2pt) node[anchor=south] {\small $x^2zw$};
\filldraw [black] (4,0) circle (2pt) node[anchor=south] {\small $y^2w^2$};
\filldraw [black] (4,1) circle (2pt) node[anchor=south] {\small $y^2zw$};
\filldraw [black] (2,2) circle (2pt) node[anchor=south] {\small $x^2w^2$};
\filldraw [black] (3,2) circle (2pt) node[anchor=south] {\small $xyw^2$};
\filldraw [black] (4,2) circle (2pt) node[anchor=south] {\small $y^2w^2$};

\filldraw[color=red, fill=red!5,  thick](5,0) -- (6,0) -- (6,1) -- (5,1) -- cycle;

\filldraw[color=red, fill=red!5,  thick](6,0) -- (7,0) -- (7,1) -- (6,1) -- cycle;

\filldraw[color=red, fill=red!5,  thick](7,0) -- (8,0) -- (8,1) -- (7,1) -- cycle;

\filldraw[color=red, fill=red!5,  thick](5,1) -- (6,1) -- (6,2) -- (5,2) -- cycle;

\filldraw[color=red, fill=red!5,  thick](6,1) -- (7,1) -- (7,2) -- (6,2) -- cycle;

\filldraw[color=red, fill=red!5,  thick](7,1) -- (8,1) -- (8,2) -- (7,2) -- cycle;

\filldraw[color=red, fill=red!5,  thick](5,2) -- (6,2) -- (6,3) -- (5,3) -- cycle;

\filldraw[color=red, fill=red!5,  thick](6,2) -- (7,2) -- (7,3) -- (6,3) -- cycle;

\filldraw[color=red, fill=red!5,  thick](7,2) -- (8,2) -- (8,3) -- (7,3) -- cycle;
\filldraw [black] (8,3) circle (2pt) node[anchor=south] {\small $y^3w^3$};
\filldraw [black] (5,0) circle (2pt) node[anchor=north] {\small $x^3z^3$};
\filldraw [black] (6,0) circle (2pt) node[anchor=south] {\small $x^2yz^3$};
\filldraw [black] (6,1) circle (2pt) node[anchor=south] {\small $x^2yz^2w$};
\filldraw [black] (5,1) circle (2pt) node[anchor=north] {\small $x^3z^2w$};
\filldraw [black] (7,0) circle (2pt) node[anchor=north] {\small $xy^2z^3$};
\filldraw [black] (7,1) circle (2pt) node[anchor=north] {\small $xy^2z^2w$};
\filldraw [black] (8,0) circle (2pt) node[anchor=south] {\small $y^3z^3$};
\filldraw [black] (8,1) circle (2pt) node[anchor=south] {\small $y^3z^2w$};
\filldraw [black] (6,2) circle (2pt) node[anchor=south] {\small $x^2yzw^2$};
\filldraw [black] (5,2) circle (2pt) node[anchor=north] {\small $x^3zw^2$};
\filldraw [black] (7,2) circle (2pt) node[anchor=north] {\small $xy^2zw^2$};
\filldraw [black] (8,2) circle (2pt) node[anchor=south] {\small $y^3zw^2$};
\filldraw [black] (6,3) circle (2pt) node[anchor=south] {\small $x^2yw^3$};
\filldraw [black] (5,3) circle (2pt) node[anchor=north] {\small $x^3w^3$};
\filldraw [black] (7,3) circle (2pt) node[anchor=north] {\small $xy^2w^3$};

\filldraw[color=red, fill=red!5,  thick](9,0) -- (10,0) -- (10,1) -- (9,1) -- cycle;
\filldraw[color=red, fill=red!5,  thick](10,0) -- (11,0) -- (11,1) -- (10,1) -- cycle;
\filldraw[color=red, fill=red!5,  thick](11,0) -- (12,0) -- (12,1) -- (11,1) -- cycle;
\filldraw[color=red, fill=red!5,  thick](12,0) -- (13,0) -- (13,1) -- (12,1) -- cycle;
\filldraw[color=red, fill=red!5,  thick](9,1) -- (10,1) -- (10,2) -- (9,2) -- cycle;
\filldraw[color=red, fill=red!5,  thick](10,1) -- (11,1) -- (11,2) -- (10,2) -- cycle;
\filldraw[color=red, fill=red!5,  thick](11,1) -- (12,1) -- (12,2) -- (11,2) -- cycle;
\filldraw[color=red, fill=red!5,  thick](12,1) -- (13,1) -- (13,2) -- (12,2) -- cycle;
\filldraw[color=red, fill=red!5,  thick](9,2) -- (10,2) -- (10,3) -- (9,3) -- cycle;
\filldraw[color=red, fill=red!5,  thick](10,2) -- (11,2) -- (11,3) -- (10,3) -- cycle;
\filldraw[color=red, fill=red!5,  thick](11,2) -- (12,2) -- (12,3) -- (11,3) -- cycle;
\filldraw[color=red, fill=red!5,  thick](12,2) -- (13,2) -- (13,3) -- (12,3) -- cycle;
\filldraw[color=red, fill=red!5,  thick](9,3) -- (10,3) -- (10,4) -- (9,4) -- cycle;
\filldraw[color=red, fill=red!5,  thick](10,3) -- (11,3) -- (11,4) -- (10,4) -- cycle;
\filldraw[color=red, fill=red!5,  thick](11,3) -- (12,3) -- (12,4) -- (11,4) -- cycle;
\filldraw[color=red, fill=red!5,  thick](12,3) -- (13,3) -- (13,4) -- (12,4) -- cycle;
\filldraw [black] (13,4) circle (2pt) node[anchor=north] {\small $y^4w^4$};
\filldraw [black] (9,0) circle (2pt) node[anchor=north] {\small $x^4z^4$};
\filldraw [black] (10,0) circle (2pt) node[anchor=south] {\small$x^3yz^4$};
\filldraw [black] (10,1) circle (2pt) node[anchor=south] {\small$x^3yz^3w$};
\filldraw [black] (9,1) circle (2pt) node[anchor=north] {\small$x^4z^3w$};
\filldraw [black] (11,0) circle (2pt) node[anchor=north] {\small$x^2y^2z^4$};
\filldraw [black] (11,1) circle (2pt) node[anchor=north] {\small$x^2y^2z^3w$};
\filldraw [black] (12,0) circle (2pt) node[anchor=south] {\small$xy^3z^4$};
\filldraw [black] (12,1) circle (2pt) node[anchor=south] {\small$xy^3z^3w$};
\filldraw [black] (13,0) circle (2pt) node[anchor=north] {\small$y^4z^4$};
\filldraw [black] (13,1) circle (2pt) node[anchor=north] {\small$y^2z^3w$};
\filldraw [black] (10,2) circle (2pt) node[anchor=south] {\small$x^3yz^2w^2$};
\filldraw [black] (9,2) circle (2pt) node[anchor=north] {\small$x^4z^2w^2$};
\filldraw [black] (11,2) circle (2pt) node[anchor=north] {\small$x^2y^2z^2w^2$};
\filldraw [black] (12,2) circle (2pt) node[anchor=south] {\small$xy^3z^2w^2$};
\filldraw [black] (13,2) circle (2pt) node[anchor=north] {\small$y^4z^2w^2$};
\filldraw [black] (10,3) circle (2pt) node[anchor=south] {\small$x^3yzw^3$};
\filldraw [black] (9,3) circle (2pt) node[anchor=north] {\small$x^4zw^3$};
\filldraw [black] (11,3) circle (2pt) node[anchor=north] {\small$x^2y^2zw^3$};
\filldraw [black] (12,3) circle (2pt) node[anchor=south] {\small$xy^3zw^3$};
\filldraw [black] (13,3) circle (2pt) node[anchor=north] {\small$y^4zw^3$};
\filldraw [black] (10,4) circle (2pt) node[anchor=south] {\small$x^3yw^4$};
\filldraw [black] (9,4) circle (2pt) node[anchor=north] {\small$x^4w^4$};
\filldraw [black] (11,4) circle (2pt) node[anchor=north] {\small$x^2y^2w^4$};
\filldraw [black] (12,4) circle (2pt) node[anchor=south] {\small$xy^3w^4$};

\filldraw [black] (0.5,-0.5) circle (0pt) node[anchor=north] {$X_1$};
\filldraw [black] (3,-0.5) circle (0pt) node[anchor=north] {$X_2$};
\filldraw [black] (6.5,-0.5) circle (0pt) node[anchor=north] {$X_3$};
\filldraw [black] (11,-0.5) circle (0pt) node[anchor=north] {$X_4$};
\end{tikzpicture}
\caption{The labelled cell complexes $X_1,X_2,X_3,X_4$ supporting the resolutions of $S/I,S/I^2,S/I^3$ and $S/I^4$ of Example \ref{squarex}. }
\label{squares}
\end{figure}

These resolutions are minimal and the first two are
$$ S\xleftarrow{\left[\begin{array}{cccc}xz& xw& yz& yw\end{array}\right]}S^4\xleftarrow{\left[\begin{array}{cccc}
-w&-y&0&0\\
z&0&-y&0\\
0&x&0&-w\\
0&0&x&z
\end{array}\right]}^4\xleftarrow{\left[\begin{array}{c}
y\\ -w\\ z\\ -x
\end{array}\right]} S\leftarrow 0$$
and
$$S\xleftarrow{\ d_1\ }S^9\xleftarrow{\ d_2\ }S^{12}\xleftarrow{\ d_3\ }S^4\leftarrow0$$
with maps
$$d_1= \left[\begin{array}{ccccccccc}
x^2z^2&x^2zw&x^2w^2&xyz^2&xyzw&xyw^2&y^2z^2&y^2zw&y^2w^2
\end{array}\right],$$
$$d_2= \left[\begin{array}{cccccccccccc}
-w&-y&0&0&0&0&0&0&0&0&0&0\\
z&0&-w&-y&0&0&0&0&0&0&0&0\\
0&0&z&0&-y&0&0&0&0&0&0&0\\
0&x&0&0&0&-w&-y&0&0&0&0&0\\
0&0&0&x&0&z&0&-w&-y&0&0&0\\
0&0&0&0&x&0&0&z&0&-y&0&0\\
0&0&0&0&0&0&x&0&0&0&-y&0\\
0&0&0&0&0&0&0&0&x&0&z&-w\\
0&0&0&0&0&0&0&0&0&x&0&z
\end{array}\right],$$
and $$d_3=\left[\begin{array}{cccc}
-y&0&0&0\\
w&0&0&0\\
0&-y&0&0\\
-z&w&0&0\\
0&-z&0&0\\
x&0&-y&0\\
0&0&w&0\\
0&x&0&-y\\
0&0&-z&w\\
0&0&0&-z\\
0&0&x&0\\
0&0&0&x\\
\end{array}\right].$$
Looking at the first two resolutions we can already see a pattern in them, which appears to continue if one computes further resolutions. Now we can use the results on the syzygy functor to show that the pattern is indeed there and the syzygies are finitely generated. 
The maps in this family are again the multiplication maps. On the cellular side they correspond to sending squares to squares. It is then not hard to see that we have covering of the cell complex for any of the complexes $X_i$ for $i\geq 2$. Then by Theorem \ref{covering} the family has finitely generated syzygy functors for all $p$.
The finite generation of syzygies holds in general for an $n$-cube, see Section \ref{cube}.
\end{example}

\section{Powers of ideals with finitely generated sygyzies}
In this section we focus on families coming from powers of ideals and finite generation of syzygies in them.
\subsection{Maximal monomial ideals of $S$}
\label{maxi}
Let $S=k[x_1,x_2,\ldots,x_n]$ be a graded polynomial ring and let $\m$ be the maximal monomial ideal. We know that the minimal resolution of $S/\m$ is supported on the $n$-simplex. In this section we want to focus on the families of cellular resolutions given by the powers of $\m$.  First we show that the resolutions $S/\m^k$, for $k>0$ are supported on a subdivided $n$-simplex.

\begin{definition}
\label{subdiv}
Let $X_n^k$ be the labelled cellular complex given by the Newton polytope of $\m^k$, i.e. the vertices are given by the exponent vectors of the monomials in $\m^k$, with subdivision with the following hyperplanes
$$H_{i,j}=\{y\in \mathbb{R}^n|y_i=j\}$$
for $0<i\leq n$ and $j\geq0$. The labels are given by the monomials in $\m^k$ and placed according to the Newton polytope. 
\end{definition}

\begin{remark}
On the level of vertices and edges, we have that any two vertices that differ by a single variable will be connected by an egde in this subdivision. 
\end{remark}

For an example of the cell complex see Figure \ref{nonminpic1}.

\begin{proposition}
The module $S/\m^k$ has a cellular resolution supported on $X_n^k$.
\end{proposition}
\begin{proof}

Fix an $n$ and the polynomial ring $S$. Note that for each $n$ we work over a different polynomial ring $S_n=k[x_1,x_2,\ldots,x_n]$. 


Let us consider the cellular complex given by $X_n^k$. 
This cell complex has zero reduced homology.

Next consider the cell complex $X_{n^k \preceq b}$ for any $b\in \mathbb{N}^n$. Due to the labelling with the least common multiple, any cell in $X_{n^k \preceq b}$ will contain its boundary and the labeling gives us that we cannot create holes into the cell complex with bounding by some $b\in \mathbb{N}^n$. Then $X_{n^k \preceq b}$ still has zero reduced homology for any $b$, and by Proposition \ref{bound} $X_n^k$ supports the cellular resolution of $\m^k$.
\end{proof}

Next we want to look at the maps between the powers. From the Example \ref{triangleex} we would expect that the maps are gain multiplication by a variable. 
\begin{proposition}
Cellular resolution morphism corresponding to multiplication by a variable in $S$ give morphism between consecutive powers of $\m$.
\end{proposition}
\begin{proof}
Let us fix a variable, say $x_i$, for the morphism. Then we know that the set map induced by it is the multiplication by $x_i$, so a vertex with a monomial  label $l$ will map to $x_il$. On the cellular side, this corresponds to the embedding of $X_n^{k-1}$ to $X_n^k$ such that $X_n^{k-1}$ covers the ``corner" with $x_i^k$.
We also know that the embedding has a corresponding chain map, hence the multiplications by a variable form a map between consecutive powers. 
\end{proof}

The existence of these desired multiplication maps gives us a linear structure on the family. Since we cannot have other multiplication maps between distinct cellular resolutions of the powers of the maximal ideal, we have the following.

\begin{proposition}
\label{linfam}
The family of cellular resolutions given by $S/\m^k$, for $k\in \mathbb{N}$, with multiplication maps is a linear family of cellular resolutions.
\end{proposition}
\begin{proof}
Given any two members of the family, say $S/\m^i$ and $S/\m^j$ with $i<j$, the possible maps between them must have $f_0$ a multiplication due to the labels. Moreover, we can split any monomial $m$ multiplication to maps given by a variable, the order may vary so the decomposition is not unique. This gives us the condition for any $f_{i,i+k}: F_i\rightarrow F_{i+k}$ there exists some consecutive morphisms such that $f_{i,i+k}=f_{i+k-1,i+k}\circ f_{i+k-2,i+k-1}\circ\ldots\circ f_{i+1,i+2}\circ f_{i,i+1}$. 
\end{proof}

Next we want to show that the subdivided simplicial complexes behave as the complexes in Example \ref{triangleex}. One can use the same method to show the tetrahedron is also covered after three subdivisions. However, drawing (or building) the cell complexes gets somewhat complicated from dimension four upwards, so we would like to have a more general proof for these coverings. For this we will make use of the following observations  from the low-dimensional cases: the square-free part is not fully covered since nothing can map to that cell. Once we only have a single square-free vertex, then the cell complex is covered by copies of one subdivision lower cell complexes. 

Note that we want to consider embeddings of cell complexes that correspond to the morphisms between $S/\m^{p-1}$ and $S/\m^p$, which are in fact the only possible embeddings.

\begin{proposition}
\label{maxcover}
Let $n$ be a positive integer and  $g:X_n^{p-1}\rightarrow X_n^p$  be an embedding of cell complexes for $p>0$ corresponding to the cellular resolution morphisms. 
Then  the embeddings of $g:X_n^{n-2}\rightarrow X_n^{n-1}$ cover $X_n^{n-1}$, and there cannot be a covering in the lower subdivisions. 
\end{proposition}
\begin{proof}
We want to consider the cell complexes as the labelled complexes to help keep track of the cells. Here we refer to the vertices by the monomial labels. 

Firstly showing that $n$ copies of $X_n^{n-3}$ do not cover $X_n^{n-2}$ follows from the square-free monomials and the multiplication maps.
Let us consider the cell formed by the square-free vertices. In the $(n-2)$ subdivision we have the degree $(n-1)$ monomials. In particular this means we have an $n$-simplex formed by these vertices. For it to be covered, we must have another cell of the same dimension mapping to it, which implies that at least one copy of lower subdivisions should cover the whole square-free simplex. We know that this cannot happen since we always have a vertex that does not contain any chosen variable. Hence it cannot be covered with a map corresponding to a multiplication by a variable. 

Next we want to show that the next subdivision is coverable and this property holds after it. 
In the cell complex the parts that are not covered by the maps can also be seen as the bounded polytope by the hyperplanes $H_{i,1}=\{y\in \mathbb{R}^n\,|\,y_i=1\}$.
In $X_n^{n-1}$ these planes intersect in a single point, namely the vertex with the label $x_1x_2\ldots x_n$. So we do not have any cells that are not covered by the embedded $X_n^{n-2}$.
Finally one wants to show that the covering continues in the higher steps. The intersection of the embeddings is formed from the vertices that have $x_1x_2\ldots x_n$ in their label. In $X_n^{i}$ for $i\geq n-1$ there are vertices with these labels, hence the intersection of all embeddings is not empty. Hence the embeddings cover the whole cell complex. 
\end{proof}

\begin{theorem}
Let $\m$ be a maximal ideal of the polynomial ring $S=k[x_1,x_2,\ldots,x_n]$. 
Then the syzygies of the modules $S/\m^p$ for $p>0$ are finitely generated.
\end{theorem}
\begin{proof}
From Proposition \ref{linfam} we have that the family of resolutions $S/\m^p$ has linear structure, and we have defined the morphisms to be only three between consecutive powers so the Hom sets are finitely generated. Then by Proposition \ref{noethrep} we have that the family has noetherian representation category. This together with Proposition \ref{maxcover} allows us to apply the 
Theorem \ref{covering} and we have that the syzygies are finitely generated. 
\end{proof}

\subsection{``Cube ideals"}
\label{cube}
In this section let $S=k[x_1,x_2,\ldots,x_{2n}]$ be a polynomial ring in $2n$ variables. 
Given an $n$-dimensional cube, label it with monomials such that along each of the $n$ edge directions we assign a pair of variables $x_i$ and $x_j$ giving $x_i$ to one end of the edges in this direction and $x_j$ in the other. The labels on the vertices are then squarefree monomials of degree $n$. See Figure \ref{squares} for an example of 2-dimensional case and Figure \ref{cubes} for the 3-dimensional labeled cube and the first subdivision.  The labelling of the cube can also be thought of assigning a pair to the parallel hyperplanes that cut out the cube, with each of the hyperplanes associated to one of the varaibles in the pair. The label on a vertex is then monomial of all the variables of the hyperplanes it sits in.

Let us denote the pairing of variables with $\pair$ that consists of  $P_1,P_2,\ldots, P_n$ where $P_i$ is the set of indices of the pair assigned to the $i$-the edge direction.  Then we can describe the set of labels on the $n$-cube as the following
$$I_{\pair}=\left\{x_{i_1}x_{i_2}\ldots x_{i_n}| i_j\in P_j\right\}.$$
Moreover this set is a square-free monomial ideal of $S$, and the $n$-cube supports the resolution of $S/I_{\pair}$.
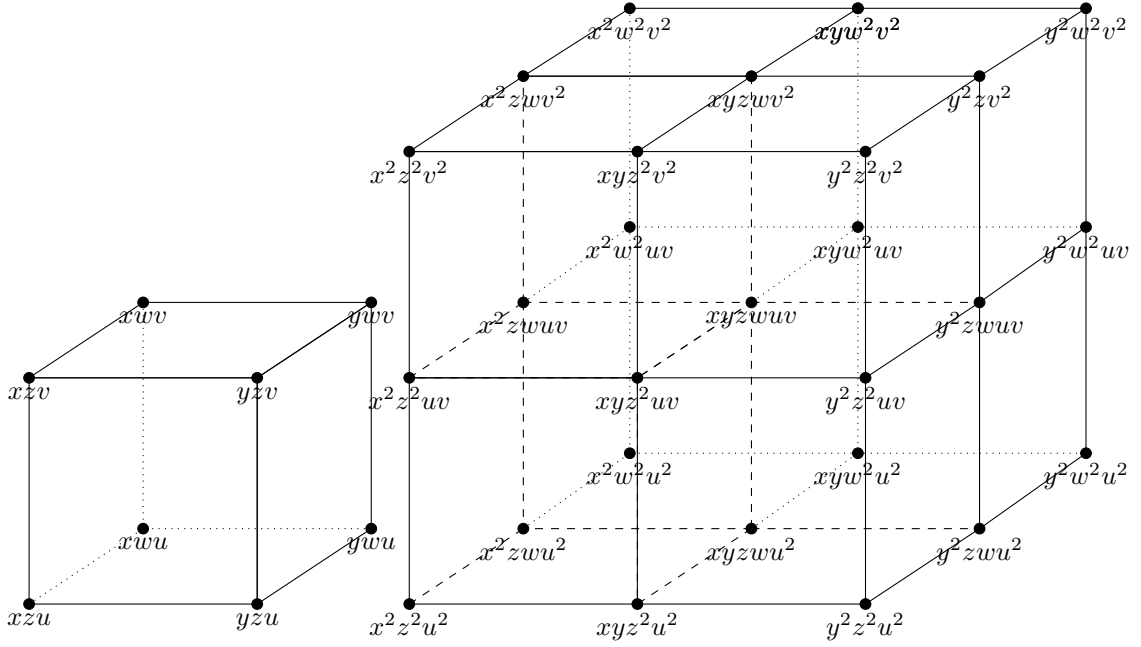
\begin{figure}
\begin{tikzpicture}
\draw[black] (0,0)--(3,0)--(3,3)--(0,3)--cycle;
\draw[black] (4.5,1)--(3,0)--(3,3)--(4.5,4)--cycle;
\draw[black] (1.5,4)--(4.5,4)--(3,3)--(0,3)--cycle;
\draw[black, dotted](0,0)--(1.5,1);
\draw[black, dotted](1.5,4)--(1.5,1);
\draw[black, dotted](4.5,1)--(1.5,1);

\filldraw [black] (0,0) circle (2pt) node[anchor=north] {$xzu$};
\filldraw [black] (3,0) circle (2pt) node[anchor=north] {$yzu$};
\filldraw [black] (3,3) circle (2pt) node[anchor=north] {$yzv$};
\filldraw [black] (0,3) circle (2pt) node[anchor=north] {$xzv$};
\filldraw [black] (4.5,1) circle (2pt) node[anchor=north] {$ywu$};
\filldraw [black] (4.5,4) circle (2pt) node[anchor=north] {$ywv$};
\filldraw [black] (1.5,4) circle (2pt) node[anchor=north] {$xwv$};
\filldraw [black] (1.5,1) circle (2pt) node[anchor=north] {$xwu$};


\draw[black] (5,0)--(8,0)--(8,3)--(5,3)--cycle;
\draw[black, dashed] (9.5,1)--(8,0)--(8,3)--(9.5,4)--cycle;
\draw[black, dashed] (6.5,4)--(9.5,4)--(8,3)--(5,3)--cycle;
\draw[black, dashed](5,0)--(6.5,1);
\draw[black, dashed](6.5,4)--(6.5,1);
\draw[black, dashed](9.5,1)--(6.5,1);
\draw[black](5,3)--(5,6)--(8,6)--(8,3)--cycle;
\draw[black] (6.5,7)--(9.5,7)--(8,6)--(5,6)--cycle;
\draw[black] (6.5,7)--(9.5,7)--(10.9,7.9)--(7.9,7.9)--cycle;
\draw[black,dashed](6.5,4)--(6.5,7);
\draw[black, dotted](6.5,4)--(7.9,5);
\draw[black, dotted](7.9,7.9)--(7.9,5);
\draw[black, dotted](10.9,5)--(7.9,5);
\draw[black, dotted](10.9,5)--(10.9,7.9);
\draw[black, dotted](10.9,5)--(9.5,4);
\draw[black, dashed](9.5,7)--(9.5,4);
\draw[black, dotted](6.5,1)--(7.9,2);
\draw[black, dotted](7.9,5)--(7.9,2);
\draw[black, dotted](10.9,2)--(7.9,2);
\draw[black, dotted](9.5,1)--(10.9,2)--(10.9,5);
\draw[black](8,0)--(11,0)--(11,3)--(8,3);
\draw[black](8,6)--(11,6)--(11,3);
\draw[black](10.9,7.9)--(13.9,7.9)--(12.5,7)--(11,6);
\draw[black](9.5,7)--(12.5,7)--(12.5,1)--(13.9,2)--(13.9,5)--(12.5,4)--(11,3);
\draw[black](13.9,7.9)--(13.9,5);
\draw[black](11,0)--(12.5,1);
\draw[black, dashed](9.5,1)--(12.5,1);
\draw[black, dashed](9.5,4)--(12.5,4);
\draw[black, dotted](10.9,2)--(13.9,2);
\draw[black, dotted](10.9,5)--(13.9,5);

\filldraw [black](5,0) circle (2pt) node[anchor=north] {$x^2z^2u^2$};
\filldraw [black](8,0) circle (2pt) node[anchor=north] {$xyz^2u^2$};
\filldraw [black](8,3) circle (2pt) node[anchor=north] {$xyz^2uv$};
\filldraw [black](5,3) circle (2pt) node[anchor=north] {$x^2z^2uv$};
\filldraw [black](9.5,1) circle (2pt) node[anchor=north] {$xyzwu^2$};
\filldraw [black](9.5,4) circle (2pt) node[anchor=north] {$xyzwuv$};
\filldraw [black](6.5,4) circle (2pt) node[anchor=north] {$x^2zwuv$};
\filldraw [black](6.5,1) circle (2pt) node[anchor=north] {$x^2zwu^2$};
\filldraw [black](5,6) circle (2pt) node[anchor=north] {$x^2z^2v^2$};
\filldraw [black](8,6) circle (2pt) node[anchor=north] {$xyz^2v^2$};
\filldraw [black](6.5,7) circle (2pt) node[anchor=north] {$x^2zwv^2$};
\filldraw [black](9.5,7) circle (2pt) node[anchor=north] {$xyzwv^2$};
\filldraw [black](10.9,7.9) circle (2pt) node[anchor=north] {$xyw^2v^2$};
\filldraw [black](7.9,7.9) circle (2pt) node[anchor=north] {$x^2w^2v^2$};
\filldraw [black](7.9,5) circle (2pt) node[anchor=north] {$x^2w^2uv$};
\filldraw [black](10.9,5) circle (2pt) node[anchor=north] {$xyw^2uv$};
\filldraw [black](10.9,7.9) circle (2pt) node[anchor=north] {$xyw^2v^2$};
\filldraw [black](7.9,2) circle (2pt) node[anchor=north] {$x^2w^2u^2$};
\filldraw [black](10.9,2) circle (2pt) node[anchor=north] {$xyw^2u^2$};
\filldraw [black](11,0) circle (2pt) node[anchor=north] {$y^2z^2u^2$};
\filldraw [black](11,3) circle (2pt) node[anchor=north] {$y^2z^2uv$};
\filldraw [black](11,6) circle (2pt) node[anchor=north] {$y^2z^2v^2$};
\filldraw [black](13.9,7.9) circle (2pt) node[anchor=north] {$y^2w^2v^2$};
\filldraw [black](12.5,7) circle (2pt) node[anchor=north] {$y^2zv^2$};
\filldraw [black](12.5,1) circle (2pt) node[anchor=north] {$y^2zwu^2$};
\filldraw [black](13.9,2) circle (2pt) node[anchor=north] {$y^2w^2u^2$};
\filldraw [black](13.9,5) circle (2pt) node[anchor=north] {$y^2w^2uv$};
\filldraw [black](12.5,4) circle (2pt) node[anchor=north] {$y^2zwuv$};

\end{tikzpicture}
\caption{A three dimensional cube and the first subdivision of the three dimensional cube labeled with the square ideal rules in variables $x,y,z,w,u,v$.}
\label{cubes}
\end{figure}

By subdividing the $n$-cube we mean an $n$-cube that consists of smaller cubes subdividing the edges into $p$ parts. Now the labeling is done with the same pairs as for the $n$-cube, but the labels moving along the directions are $x_i^p$, $x_i^{p-1}x_j,\ldots,x_ix_j^{p-1},x_j^p$. Alternatively,  the subdivisions can be thought of cutting by $(p-1)$  hyperplanes parallel to the pairs of hyperplanes bounding to the cube. Then each of these hyperplanes is given a monomial $x_i^p$, $x_i^{p-1}x_j,\ldots,x_ix_j^{p-1},x_j^p$ and the monomial on the vertex is again the monomial given by all the variables of the hyperplanes where the vertex is. 
Denote the $n$ cube with subdivision to $p$ edge sections by $C_n^p$.
\begin{proposition}
Fix a pairing $\pair$ on the edge directions. 
The cell complex $C_n^p$ supports the resolution of $S/I^n_p$ where $I_{\pair}=\left\{x_{i_1}x_{i_2}\ldots x_{i_n}| i_j\in P_j\right\}$.
\end{proposition}
\begin{proof}

Let us consider the cell complex $(C_{n}^p)_{\preceq{b}}$ for any $b\in\mathbb{N}^n$. With our chosen labelling, bounding a single coordinate in the exponent vector results in cutting the subdivided cube along one of the hyperplanes dividing it, unless all vertices are included or none. If the vector $b$ has entries $b_i$ and $b_j$ for some pair $ij$ such that $b_i+b_j<p$, then $(C_{n}^p)_{\preceq{b}}$ is an empty complex. Now we may assume that in $b$ the entries for each pair are such that we do not get an empty complex. Then  $(C_{n}^p)_{\preceq{b}}$  will still consist of cubes, and will be contractible. 
Thus it will be acyclic.
\end{proof}

Next we want to look at the maps between the powers of $I_{\pair}$ for a fixed $\pair$. In particular, we want the maps from (subdivided) cubes to (subdivided) cubes that maps $m$-cells to $m$-cells for all $m$. These maps correspond to embedding the 
$C_n^{p-1}$ to $C_n^p$, and to make the labels well behaved for that the algebraic side of the map has to be a multiplication by one of the monomials of $C_n^1$. These maps are enough to give a covering of the cell complexes.

\begin{proposition}
\label{cubecover}
Fix a dimension $n$, then $C_n^{1}$ covers $C_n^p$ for $p\geq2$. 
\end{proposition}
\begin{proof}
We want to look at the cell maps coming from the multiplication by elements from $C_n^1$. So we have $2^n$ multiplication maps between consecutive $C_n^{p-1}$ and $C_n^p$. 
On the cell complex side these correspond to embeddings. 
The first subdivision to $C_n^{2}$ consist of $2^n$ cubes.

The first cube $C_n^{1}$ maps to every $n$-cube in $C_n^p$.
This follows from every subdividing cube is bounded by consecutive hyperplanes, i.e. ones that have varaibles $x_i^mx_j^{m'}$ and $x_i^{m-1}x_j^{m'+1}$ associated to them, thus each cube has labels that are multiple of the labels of $C_n^{1}$ by some monomial. Hence there is a map taking $C_n^{1}$ to every $n$-cube in $C_n^p$ and we have a covering. 
\end{proof}

Denote the family of cellular resolutions  coming from the resolutions supported on the labeled $n$-cube by $\F_C$. As with the maximal ideals we want to restrict the morphisms such that for each chain map there is only one morphism.

\begin{proposition}
For a fixed $n$ the family $\F_C$  has finitely generated syzygies.
\end{proposition}
\begin{proof}
The family $\F_C$ has $2^n$ maps between consecutive powers, which are multiplications by the monomials in $I_{\pair}$. Now all the other maps are compositions of these, so the family satisfies the condition of being linear. Moreover since the Hom sets are finite, we can apply the Proposition \ref{noethrep} to get that $\repS(\F_C)$ is noetherian.
Now using the Proposition \ref{cubecover} we can apply Theorem \ref{covering} to get the result.
\end{proof}

\subsection{Equigenerated ideals}
Next we want to look at ideals where the generators have the same degree.
\begin{definition}
A monomial ideal is said to be \emph{equigenerated} if all the generators have the same degree.
\end{definition}
The maximal ideals, cube ideals, and their powers are all examples of equigenerated ideals. We know the behaviour of these from the previous sections, so the focus now is on equigenerated ideals that are neither the maximal ideal nor cube ideal or a power of one them.

 A useful observation one can make is that given an equigenerated ideal $I$ in $n$ variables in degree $d$ we have inclusion of $I$ to $\m_n^d$. If one takes the power of $I$, this new ideal will have degree $2d$, and the possible maps between $X_n^d$ and $X_n^{2d}$ are given by multiplications corresponding to the monomials in $\m_n^d$.
Another point is that we can give a new orientation on the cell complex to emphasize the parts we want. In the equigenerated case a natural choice is to consider the subcomplex in $X_n^d$ coming from deleting all the vertices that are not in the generators of the equigenrated ideal $I$. Call this subcomplex $X_I^d$.   We can write the resolution of $S/\m_n^d$ as follows
$$ 0\leftarrow S\xleftarrow{d_1}F_1\oplus F'_1\leftarrow\dots\leftarrow F_i\oplus F'_i\xleftarrow{d_{i+1}}F_{i+1}\oplus F'_{i+1}\leftarrow\dots$$
where the generators of $F_i$ correspond to the cells in $X_I^d$ and $F'_i$ has generators corresponding to the other cells in $X_n^d$. The map $d_i$ can be written as a block matrix $\left[\begin{array}{cc}
   A  &  B\\
    0 & C
\end{array}\right]$ where $A$ corresponds to the map of $F_{i+1}$ to $F_i$, $B$ is the map from $F'_{i+1}$ to $F_i$, and $C$ is the map from $F'_{i+1}$ to $F'_i$.

First let us consider the equigenerated ideals that given by the bounds on the exponent vector. Each vertex label can be identified with a degree vector $a=(a_1,a_2,\ldots,a_n)$ where $a_n$ is the degree of $x_n$ in the monomial. Let $b\in \mathbb{N}^n$, then $a\preceq b$ if $a_i\leq b_i$. Denote by $I_{\preceq b}$ the ideal of monomials bound by $b$ in a fixed $\m_n^d$.

\begin{proposition}
\label{boundcover}
The cellular resolution of $I_{\preceq b}^m$ is supported on the cell complex $X_{n,\preceq mb}^{md}$, and the cell complexes $X_{n,\preceq kb}^{kd}$ have coverings for large enough $k$.

\end{proposition}
\begin{proof}
From the Proposition \ref{bound} we know that the cell complex given by bounding $X_n^d$ with the given vector is an acyclic cell complex. 
Furthermore, if we bound $X_{n,\preceq mb}^{md}$ with any vector $a\preceq b$, it is the same as bounding $X_n^d$, i.e. giving an acyclic cell complex, and if we choose some $c$ such that $b\preceq c$, then it gives us the whole complex $X_n^d$. Thus bounding $X_{n,\preceq mb}^{md}$ with any vector will give an acyclic cell complex and by Proposition \ref{bound} it supports the cellular resolution of the module given by the monomial labels of $X_{n,\preceq mb}^{md}$.

The monomials generating $I_{\preceq b}^m$ appear as vertex labels in $X_{n,\preceq mb}^{md}$, since any monomial in $I_{\preceq b}^m$ has an exponent vector with entries at most $m$ times the entries of monomials in $I_{\preceq b}$. 
We want to show that there are no other vertices in the $X_{n,\preceq mb}^{md}$. To do this we assume there is a vertex that has a monomial label which does not come from generators of  $I_{\preceq b}^m$. The monomial label can be written as a product of $m$ monomials of degree $d$. We can then write the exponent vector of this monomial as $\left(\sum_{i=1}^m a_{i1},\sum_{i=1}^m a_{i2},\ldots,\sum_{i=1}^m a_{in}\right)$ where $(a_{i1},a_{i2},\ldots,a_{in})$ is the exponent vector of the $i$-th monomial $a_i$. If this monomial is not a generator in $I_{\preceq b}^m$ then at least one of the $a_i$ has to be outside of $I_{\preceq b}$, that is there exists at least one $a_{ij}> b_j$ for some $i$ and $j$. We can then focus on those $j$'s that contain entries $>b_j$, and use $k$ to denote the number of the monomials with $j$-th entry $>b_j$. Then we have 
$$mb_j\geq\sum_{i=1}^m a_{ij}=\sum_{i=0,a_{ij}>b}^m a_{ij}+\sum_{i=0, a_{ij}\leq b}^m a_{ij},$$
we can divide by $m$ since it is a positive integer and get
$$b_j\geq\frac{\sum_{i=0,a_{ij}>b}^m a_{ij}}{m}+\frac{\sum_{i=0, a_{ij}\leq b}^m a_{ij}}{m}.$$
On the first sum we can apply the inequality $\sum_{i=1}^m\frac{\lambda_m}{m}\geq\prod_{i=1}^m(\lambda_i)^{1/m}$ and use that each of the terms is strictly greater than $b_i$ to get 
$$\frac{\sum_{i=0,a_{ij}>b}^m a_{ij}}{m}\geq \prod_{i=0,a_{ij}>b}^m (a_{ij})^{1/k}>\prod_{i=0}^k b_j^{1/k}=b_j.$$
Then combining this with the earlier inequality we have
$$b_j\geq b_j+\frac{\sum_{i=0, a_{ij}\leq b}^m a_{ij}}{m},$$
Which is a contradiction and hence we cannot have any vertices with labels not coming from $I_{\preceq b}^m$ in $X_{n,\preceq mb}^{md}$. Moreover, since $X_{n,\preceq mb}^{md}$ satisfies the acyclicity conditions it supports the cellular resolution of $I_{\preceq b}^m$.

Furthermore, by having all the vertex labels being products of monomials from $I_{\preceq b}$ implies that we cannot have cells outside $X_{n,\preceq kb}^{kd}$ mapping into $X_{n,\preceq mb}^{md}$ with $1\leq k\leq m-1$. We know that the cell complexes $X_{n}^{td}$ have covering for large enough $t$, thus $X_{n,\preceq mb}^{md}$ is also covered for large enough $m$.
\end{proof}

The cellular resolutions of $I_{\preceq b}^m$ form a linear family with their structure of maps and by combining Proposition \ref{boundcover} and Theorem \ref{covering} we get the following: 

\begin{corollary}
\label{boundfg}
The family of cellular resolutions of $I_{\preceq b}^m$ has finitely generated syzygy representation.
\end{corollary}




We want to consider connected equigenerated ideal $I$  such that they have a cellular resolutions supported on the cell complex coming from $X_n^d$ by removing all vertices with a label not in $I$ and higher dimensional cells containing those. Denote this cell complex by $X_I^d$.

\begin{proposition}
If $I$ is an equigenerated ideal in $n$ variables and degree $d$ such that I has a cellular resolution supported on $X_I^d$ and  the powers of $I$ are supported on $X_{I^m}^{md}$, then the family of cellular resolutions given by them has finitely generated syzygies. 
\end{proposition}
\begin{proof}
We will show that the syzygies are finitely generated using subfunctors of $s_d$.
Let $(F_I^i)_d$ denote the $d$-th free module in the $i$-th power. 

Let $\maxfam$ denote the family of cellular resolutions of the powers of the maximal ideal and let $F_i$ denote the resolution of the $i$-th power of the maximal ideal.
First consider a representation of $\maxfam$ defined as follows
$$\Phi^I_d: \maxfam \rightarrow \smod$$
sending $F_i$ to $(F_I^i)_d$.
Now we have that $\Phi^I_d(F_i)$ is a direct summand of $s_d(F_i)$, so $\Phi^I_d$ is a subfunctor of $s_p$.

Moreover we can considered the representation$\phi^I_d: \maxfam \rightarrow \smod$ sending $F_i$ to the $d$-th syzygy module of the resolution of $S/I^i$. Since the syzygy module is a submodule of the $(F_I^i)_d$ we have that $\phi^I_d$ is a subfunctor of $\Phi^I_d$, and so it is a subfunctor of $s_d$ as well. 
Then by noetherianity we know that $\phi^I_d$ is finitely generated for any $d$ as $s_d$ is finitely generated for any $d$. Thus we get that the family consisting of powers fo $I$ has finitely generated syzygies.  
\end{proof}






\section{Edge ideals}
In this section we focus on the cellular resolutions that are coming from edge ideals. Through out this section let $G$ denote a simple graph with $n$ vertices and $m$ edges unless otherwise specified. If $v_1$ and $v_2$ are vertices then we denote the edge between them by $v_1v_2$ if it exists. We denote the set of edges with $E(G)$ and the set of vertices with $V(G)$.
\begin{definition}
Let $G$ be a simple graph with vertices numbered by $v_i$, $1\leq i\leq n$. The \emph{edge ideal} associated to $G$ is a monomial ideal in the polynomial ring $S=k[x_1,x_2,\ldots,x_n]$ defined as
$$I_G=(x_ix_j| v_iv_j \in E(G)).$$
\end{definition}
If we consider edge ideals coming from multiple graphs with numbers of vertices $n_1,n_2,\ldots,n_r$ then we take the polynomial ring with $\max\{n_1,n_2,\ldots,n_r\} $ variables.

The possible maps between cellular resolutions coming from edge ideals are particularly cut down by the condition that the first component of the chain map $f_0$ has to be a $1\times1$ matrix.  As noted before this leaves only multiplication as an option and is also describing what the cellular map does on the labels. The degree of the generators is always 2 for edge ideals, hence we cannot have multiplication with any element of $S$ that has degree greater 0. Thus we get that $f_0$ has to be identity map and on the level of vertices the cellular map is an embedding. 

\begin{example}
\label{exgraph}
Let $S=k[x,y,z,w]$ be a graded polynomial ring.
Consider the edge ideals $I_{G_1}=(xy,yz,zw)$ and $I_{G_2}=(xy,xw,yz,zw)$ coming from the graphs $G_1$ and $G_2$ in the Figure \ref{graphex}.

\begin{figure}
\begin{center}
\begin{tikzpicture}
\draw[black](0,2)--(2,2)--(2,0)--(0,0);
\draw[black](3,2)--(5,2)--(5,0)--(3,0)--cycle;

\filldraw [black](0,2) circle (2pt) node[anchor=south] {$1$};
\filldraw [black](2,2) circle (2pt) node[anchor=south] {$2$};
\filldraw [black](2,0) circle (2pt) node[anchor=north] {$3$};
\filldraw [black](0,0) circle (2pt) node[anchor=north] {$4$};
\filldraw [black](3,2) circle (2pt) node[anchor=south] {$1$};
\filldraw [black](5,2) circle (2pt) node[anchor=south] {$2$};
\filldraw [black](5,0) circle (2pt) node[anchor=north] {$3$};
\filldraw [black](3,0) circle (2pt) node[anchor=north] {$4$};
\end{tikzpicture}
\end{center}
\caption{Graphs $G_1$ and $G_2$ of Example \ref{exgraph}.}
\label{graphex}
\end{figure}
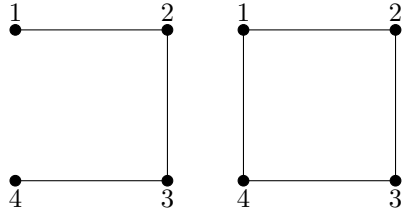

The resolutions of $S/I_{G_1}$ and $S/I_{G_2}$ are
$$ S\xleftarrow{\left[ \begin{array}{ccc} 
xy& yz& zw
\end{array}\right]} S^3\xleftarrow{\left[\begin{array}{cc}
-z&0\\
x&-w\\
0&y
\end{array}\right]} S^2\leftarrow 0$$
and
$$S\xleftarrow{\left[ \begin{array}{cccc} 
xy&xw& yz& zw
\end{array}\right]} S^4\xleftarrow{\left[\begin{array}{cccc}
-w&-z&0&0\\
y&0&-z&0\\
0&x&0&-w\\
0&0&x&y
\end{array}\right]}^4\xleftarrow{\left[\begin{array}{c}
z\\ -w\\ y\\ -x
\end{array}\right]} S\leftarrow 0$$
Both of these resolutions are cellular, with cell complexes $X_1$ and $X_2$ respectively. These cell complexes are shown in the Figure \ref{gcell}. A possible cellular resolution map between the two is given by taking a cellular map $g$ embedding the line into the square with $\varphi_g=\operatorname{id}$, and taking a chain map with $f_0=\operatorname{id}$, $f_1=\left[\begin{array}{ccc}
    1 &0&0  \\
    0 &0&0\\
    0&1&0\\
    0&0&1
\end{array}\right]$, $f_2=\left[\begin{array}{cc}
     0&0  \\
     1& 0\\
     0&0\\
     0&1
\end{array}\right]$, and $f_3=0$. This then forms a pair of compatible maps that give us the cellular resolution morphism. This is the only possible map between the two resolution (not counting the possible choices for the cellular map).

\begin{figure}
\begin{center}
\begin{tikzpicture}
\draw[color=red,  thick](0.5,0)--(2,1)--(0.5,2);
\filldraw[color=red, fill=red!5,  thick](3,2)--(5,2)--(5,0)--(3,0)--cycle;

\filldraw [black](0.5,0) circle (2pt) node[anchor=north] {$x_3x_4$};
\filldraw [black](2,1) circle (2pt) node[anchor=north] {$x_2x_3$};
\filldraw [black](0.5,2) circle (2pt) node[anchor=south] {$x_1x_2$};
\filldraw [black](3,2) circle (2pt) node[anchor=south] {$x_1x_2$};
\filldraw [black](5,2) circle (2pt) node[anchor=south] {$x_2x_3$};
\filldraw [black](5,0) circle (2pt) node[anchor=north] {$x_3x_4$};
\filldraw [black](3,0) circle (2pt) node[anchor=north] {$x_1x_4$};
\end{tikzpicture}
\end{center}
\caption{Cell complexes supporting the resolutions of $S/I_{G_1}$ and $S/I_{G_2}$.}
\label{gcell}
\end{figure}
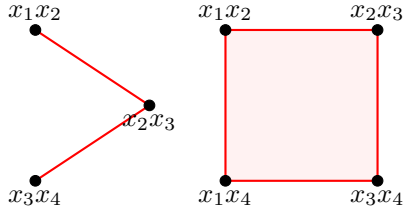
\end{example}

\subsection{Powers of edge ideals}
\subsubsection{Edge ideals of paths}
\label{paths}
One of the possible things to look at with edge ideals is of course their powers. In \cite{E10} Engstr\"om and Noren showed that the powers of edge ideals of paths have minimal cellular resolutions and gave an explicit description of a cellular resolution that is close to minimal.

A \emph{n-path} is a graph with $n$ vertices and edges 12,23,$\ldots,(n-1)n$, and it is denoted with $\pathn$.
We denote the edge ideal of path $\pathn$ with $\pathidealn$.
A \emph{Newton polytope} of a monomial ideal $I$ is a polytope $\operatorname{Newt}(I)$ given by the pan of the exponent vectors of the generators of $I$.
 
\begin{definition}[\cite{E10}, Def 3.1]
Let $\operatorname{Newt}(\pathidealn^d)$ be the Newton polytope of the ideal $\pathidealn^d$, and define the hyperplanes
$$H'_{i,j}=\left\{{\bf y}\in\mathbb{R}^n|\sum^{\lfloor (i-1)/2\rfloor}_{k=0}y_{i-2k}=j\right\}$$
and
$$H_{i,j}=\left\{{\bf y}\in\mathbb{R}^n|y_i=j\right\}$$
for $0<i\leq n$ and $0\leq j$. The the subdivision of $\operatorname{Newt}(\pathidealn^d)$ by all $H_{i,j}$ is called $Y_n^d$ and the subdivision by all $H_{i,j}$ and $H'_{i,j}$ is called $Z_n^d$.

\end{definition}

\begin{proposition}[\cite{E10}, Thm 5.6 ]
The cell complexes $Y^d_n$ and $Z_n^d$ both support a cellular resolution of $I^d_{P_n}$.
\end{proposition}

Now we can study the powers of the edge ideals of chains as a family. One way to define this family is to take cellular resolutions in our family are taken to be the ones supported on $Y^d_n$ and morphism are monomial multiplications, with ones between consecutive powers given by $I_{P_n}$. As with the families in earlier sections we aim to show that the syzygy functor for this family is finitely generated. For this purpose we want to make use of the observation that $Z^d_n$ and $Y_n^d$ are $X_n^d$ with some cells deleted and cut by few extra hyperplanes. Define the following cell complex that is a subdivision of $Z^d_n$.



\begin{definition}
Let $\overline{Z^d_n}$ be subdivision of $Z^d_n$ given by the hyperplanes
$$\overline{H}_{i,j}=\left\{{\bf y}\in\mathbb{R}^n|\sum^{\lfloor (i-1)/3\rfloor}_{k=0}y_{i-3k}=j\right\},$$
for $0<i\leq n$ and $0\leq j$.
\end{definition}

\begin{proposition}
The cell complex  $\overline{Z^d_n}$ supports a cellular resolution of $I^d_{P_n}$.
\end{proposition}
\begin{proof}
The subdivision $\overline{Z^d_n}$ adds no new vertices to $Z^d_n$.
Moreover, the labels of the cell complex $\overline{Z^d_n}$ come from lattice points and bounding by some vector  {\bf b} will give us either convex subcomplex or an empty subcomplex. Thus it will contractible and acyclic and giving us that it supports the resolution by Proposition \ref{bound}.
\end{proof}

\begin{proposition}
For  a fixed $n$ the family of cellular resolutions of powers of path ideals $P_n$ has finitely generated syzygies. 
\end{proposition}
\begin{proof}
The cell complexes are refinement of the cell complex we encountered for the maximal ideals. We know these are covered eventually, thus we get covering for the complexes since we have only four cellular maps that behave the same to the cellular maps with maximal ideals.

The powers form a linear family which follows from the possible multiplication maps. Then we have noetherianity of the representations and get that the syzygy functor is finitely generated as we have covering. 
\end{proof}

\subsubsection{The complete graph}
In this section we want to look at the powers of the edge ideals of the complete graph $K_n$ on $n$ vertices. We want to look at the resolutions of the powers of edge ideals of complete graphs. 

The complete graph on $n$ vertices is a graph that contains an edge between any two vertices. This then implies that the edge ideal of the graph $K_n$, denoted by $I_{K_n}$, is given by all pairs $x_ix_j$ with $1\leq i<j\leq n$. These monomials are precisely all the degree 2 monomials of $\m_n^2$ that are bounded by the vector $(1,1,\ldots,1)$. This allows us to apply the earlier results on the equigenerated ideals bounded by some vector.

\begin{proposition}
Let $\F$ be family of cellular resolutions coming from the powers of $I_{K_n}$, then the syzygies of the family are finitely generated.
\end{proposition}
\begin{proof}
This follows directly from the Corollary \ref{boundfg} by setting ${\bf b}=(1,1,\ldots,1)$ and $d=2$.
\end{proof}
\subsection{Booth-Lueker graphs}
\label{blideals}
In this section we want to look at results relating to the edge ideal of the Booth-Lueker graph. In \cite{bs} the formulas for the Betti numbers and other invariants of the resolution were established in terms of the graph's number of edges and vertices. The cellularity of the resolution was not studied in this paper. First we recall the Booth-Lueker edge ideal from \cite{bs} and some definitions from \cite{af}.

For this section let $G$ be a simple graph and let $I_G$ be the edge ideal of this graph. 
\begin{definition}
For any graph G let BL(G) be the graph with vertex set $V (G)\cup E(G)$ and edges uv for every pair of vertices in G and ue for every vertex u incident to an edge e in G. We call BL(G) the Booth–Lueker graph of G.
\end{definition}
Let us denote the edge ideal of the Booth-Lueker graph with $BL(I_G)$.

\begin{definition}[\cite{af}, Def 2.1]
Let $I$ be a monomial ideal. The ideal $I$ has \emph{linear quotients} if there is an ordering of the generators $(m_1,m_2,\ldots,m_k)$ such that the colon ideal $(m_1,m_2,\ldots,m_{j-1}):m_j$ is generated by some subset of variables for each $j$.
\end{definition}
\begin{definition}[\cite{af}, Def 2.2]
Let $I$ be a monomial with linear quotients for some order $(m_1,m_2,\ldots,m_k)$. The $\set$ of a generator is defined to be $\set(m_j)=\{k\in [n]\,|\, x_k\in <m_1,m_2,\ldots,m_{j-1}>:m_j\}$.
\end{definition}
\begin{definition}[\cite{af}, Def 2.4]
Let $I$ be a monomial ideal and let $M(I)$ denote the set of all monomials in $I$ and let $G(I)$ denote the set of generators of $I$. The \emph{decomposition function} of $I$ is a map from $b:M(I)\rightarrow G(I)$ defined as $b(m)=m_j$ where $j$ is the smallest index such that $m\in<m_1,m_2,\ldots,m_j>$.

The decomposition function $b$ is \emph{regular} if for every $m\in G(I)$ and every $t\in \set(m)$ we have  $\set(b(x_tm))\subseteq \set(m)$.
\end{definition}
These definitions then allows us to state the main theorem from \cite{af}.

\begin{theorem}[\cite{af}, Thm 3.10]
\label{afmain}
Suppose that ideal $I$ has linear quotients with respect to some ordering $(m_1,\ldots,m_k)$ of the generators, and suppose that $I$ has a regular decomposition function. Then the minimal resolution of $I$ obtained as an iterated mapping cone is a cellular and supported on a regular CW-complex. 
\end{theorem}

Before applying the above definitions and theorems we make some observations about the Booth-Lueker ideals.
\begin{proposition}
\label{lqr}
If $I$ is an edge ideal of a simple graph, then the ideal $BL(I)$ of the Booth-Lueker graph has linear quotients and a regular decomposition function.
\end{proposition}
\begin{proof}
Let $I$ be an edge ideal of a graph with $n$ vertices.  
The Booth-Lueker ideal of $I$ is given by 
$$
BL(I)=\left(x_1x_2,x_1x_3,\ldots,x_1,x_n,x_2x_3,\ldots,x_2x_n,\ldots,x_{n-1}x_n, \{x_iy_k,x_jy_k\, |\, ij \textrm{ is the k-th edge}\}\right).$$
First we want to show that the ideal has linear quotients for the ordering given above. 
Note that we have chosen our ordering such that for a monomial $x_ix_j$ we have $i<j$. If we consider the colon ideal of the form 
$<x_1x_2,\ldots,x_ix_{j-1}>:x_ix_j$
we can compute it to get that it has generators $x_k$ with $k<i$ or $i<k<j$. Thus until the $y$ variables the colon ideal satisfies the condition of being generated by variables. 

Let us look at the colon ideals of the form $$<x_1x_2,\ldots,x_{n-1}x_{n},x_{i_1}y_1,x_{j_1}y_1,\ldots,x_{i_t}y_t>:x_{j_t}y_t,$$ again one can compute the generators of it and these are $x_k$ with $k\neq j_{t}$and $y_i$ corresponding to an edge with $x_{i_{t}}$ with $i<t$.
Finally the last type of colon ideal we can have is 
$$<x_1x_2,\ldots,x_{n-1}x_{n},x_{i_1}y_1,x_{j_1}y_1,\ldots,x_{i_t}y_t,x_{j_t}y_t>:x_{i_{t+1}}y_{t+1}$$ and as before computing the generators gives $x_k$ with $k\neq i_{t+1}$ and $y_i$ corresponding to an edge with $x_{i_{t+1}}$ with $i<t+1$. Thus we get that $BL(I)$ has linear quotients with respect to to the order given above. 

Using the order given in the defintion of linearisation, the $\set$s of the generators can be explicitly computed to obtain 
$$\set(x_ix_j)=\{1,\ldots,i-1,i+1,\ldots,j-1\}$$
$$\set(x_iy_k)=\{1,\ldots,i-1,i+1,\ldots, n\}\cup\{n+t | x_iy_t \textrm{is a generator of }BL(I)\}$$
Note that if $t\in\set(x_ix_j)$ then $t<j$.
Next we want to check the regularity of the decomposition function. For the generators of the form $x_ix_j$ we get the following from the multiplication by variables given by the $\set(x_ix_j)$:
$$\set(b(x_tx_ix_j))=\set(x_ix_t)=\left\lbrace\begin{array}{cc}
\{1,2,\ldots,t-1,t+1,\ldots,i-1\} & \mathrm{if }\ t<i\\
\{1,2,\ldots,i-1,i+1,\ldots,t-1\} & \mathrm{if }\  i>t
\end{array}\right. .$$
It is clear that both of the sets above are contained in $\set(x_ix_j)=\{1,\ldots,i-1,i+1,\ldots,j-1\}$.

Next let us consider the regularity for the generators of the form $x_iy_k$. In this notation $x_{n+j}=y_j$.
We can divide the computation of $\set(b(x_tx_iy_k))$ to two cases, firstly when $x_ix_t$,  for any $t>n$,  is not a generator, and secondly when $x_ix_t$, for at least one $t>n$, is a generator. In the first case the computation gives us the same result as with $x_ix_j$.
In the second case we further subdivide to $t<n$ and $t>n$. 
\begin{itemize}
\item[$t<n$]When $t<n$ we have $\set(b(x_tx_iy_k))=\set(x_tx_i)$ which is $\{1,2,\ldots,t-1,t+1,\ldots,i-1\}$ or $\{1,2,\ldots,i-1,i+1,\ldots,t-1\}$. Both are contained in the set $\{1,\ldots,i-1,i+1,\ldots, n\}\cup\{n+t | x_iy_t \textrm{is a generator of }\lin(I)\}$.

\item[$t>n$] When $t>n$ we have $\set(b(x_tx_iy_k))=\set(x_tx_i)=\{1,\ldots,i-1,i+1,\ldots, n\}\cup\{n<r<t | x_ix_r \textrm{is a generator of }\lin(I)\}$. The first part is clearly contained in $\set(x_iy_k)$, further more we have that $\{n<r<t | x_ix_r \textrm{is a generator of }\lin(I)\}\subseteq\{n+t | x_iy_t \textrm{is a generator of }\lin(I)\}$ since $x_t$ must be bofore $y_k$ in the ordering of the variables.
\end{itemize}
Thus we get that the decomposition function satisfies the definition of regularity.
\end{proof}

\begin{corollary}
Every ideal of a Booth-Lueker graph has a minimal cellular resolution coming from the mapping cone resolution construction.
\end{corollary}
\begin{proof}
By Proposition \ref{lqr} the edge ideal $BL(I_G)$ of a Booth-Lueker graph has linear quotients and a regular decomposition function. Thus we can apply the Theorem \ref{afmain} to $BL(I_G)$ and get that $BL(I_G)$ has a cellular resolution made with mapping cones.
\end{proof}

Let us consider the functor from a family of edge ideal resolutions to Booth-Lueker ideal resolutions. Given an edge ideal, we can define the Booth-Lueker ideal, and this gives us a way to send a cellular resolution of an edge ideal to a cellular resolution to a corresponding Booth-Lueker edge ideal. The morphisms on edge ideals then give well defined morphisms on the Booth-Lueker ones. For the rest of this section we will denote the variable $y_k$ belonging to the $k$-th edge between $x_i$ and $x_j$ by $y_{ij}$

Let  $F_I$ and $F_J$ be two cellular resolutions that belong to edge ideals $I$ and $J$, and suppose that we have a map between $F_I$ and $F_J$. From our earlier observations this map has to be an embedding. Let us consider the resolutions for the Booth-Lueker ideals $BL(I)$ and $BL(J)$. The generators of $I$ and $J$ are contained in these ideals and moreover we have an embedding of the generators of $BL(I)$ to $BL(J)$. The cell complexes supporting  Booth-Lueker ideal resolutions coming from mapping cones are formed of a piece that corresponds to the complete graph. This can be taken to be minimal or the non-minimal one contained in the cell complex supporting a maximal ideal, and $n-1$ or higher dimensional cells from adding the vertices containing $y_{ij}$. Now we may assume that the ordering on the ideals $BL(I)$ and $BL(J)$ is such that any monomials from vertices in $BL(J)$ that are not in $BL(I)$ are ordered last. Then we have an embedding of the cell complex supporting the resolution of $BL(I)$ to the cell complex supporting a resolution of $BL(J)$ if the corresponding graphs have the same number of vertices.  If the underlying graphs have a different number of vertices, we can still embed the cell complex supporting $BL(I)$ to $BL(J)$ in the mapping cone construction gives subcomplex at each step and they are glued on in the same way for both $BL(I)$ and $BL(J)$, assuming one is consistent with the choices in building the cell complex.

Let $\edgecat$ denote the category of cellular resolutions coming from edge ideals of graphs with at most $n$ vertices, and let $\edgecatbl$ denote the category of cellular resolutions of coming from edge ideals of graphs with at most $n+\frac{n(n-1)}{2}$ vertices. If we take an edge ideal and go to the Booth-Lueker ideal of it, the resolutions will move from $\edgecat$ to $\edgecatbl$. The resolutions in $\edgecat$ are defined over the polynomial ring with $n$ variables, and we will use the variables $x_1,\ldots,x_n$ for this polynomial ring. When we use these categories in the context of having a functor the variables in the polynomial ring for $\edgecatbl$ are denoted by $$x_1,\ldots,x_n,y_{12},y_{13},\ldots,y_{1n},y_{23},\ldots,y_{(n-1)n}.$$
\begin{definition}
Let $\edgecat$ denote the category of cellular resolutions coming from edge ideals of graphs with at most $n$ vertices and $m$ edges in $k[x_1,\ldots,x_n]$. Then define the functor
$$\operatorname{BL}:\edgecat\rightarrow \edgecatbl$$
by sending cellular resolution $F_G$ to a minimal resolution of Booth-Lueker edge ideal of $G$ $F_{BL(G)}$. The functor $\operatorname{BL}$ takes an embedding of cellular resolutions to an embedding on the Booth-Lueker resolutions.
\end{definition}

This functor can be restricted to families of cellular resolutions in which case it takes a family over with individual resolutions belonging to $\edgecat$ to a family where the resolutions are in $\edgecatbl$. 

Recall that in Section \ref{repstab} we defined a property for functors called property (F). Next we show that in this setting of edge ideals the $\operatorname{BL}$ functor satisfies this property.
\begin{proposition}
The restriction of the functor $\operatorname{BL}$ between families satisfies the property (F).
\end{proposition}
\begin{proof}
Let $\F$ denote a family of cellular resolutions coming from edge ideals and let  $\operatorname{BL}(\F)$ denote the family of cellular resolutions obtained by taking each resolution $F_I$ in $\F$ to a minimal cellular resolution of the associated Booth-Lueker ideal $F_{BL(I)}$. 

Let $F_{BL(I)}$ be an arbitrary cellular resolution in $\operatorname{BL}(\F)$. We know that there is an ideal $I$ whose Booth-Lueker ideal $BL(I)$ is, and we can choose the resolution $F_I$ as our finite set of elements in $\F$. Then we get a morphism $F_{BL(I)}\rightarrow F_{BL(I)}$ that is just the identity, denoted by $\operatorname{id}$. Now pick any $F_J$ in $\F$, then the morphism from $F_{BL(I)}\rightarrow \operatorname{BL}(F_J)$ is either an embedding or there is no morphism. Let us assume we have an embedding and call it $g$. Then picking the embedding $f$ between $F_I$ and $F_J$ gives that $\operatorname{BL}(F)$ is an embedding of $F_{BL(I)}$ to $\operatorname{BL}(F_J)$. Since we only have one possible morphism between two different cellular resolutions coming from different edge ideals, we have that $g=\operatorname{BL}(F)=\operatorname{BL}(f)\circ\operatorname{id}$.
\end{proof}

This means we can pullback information on the representation from $\operatorname{BL}$ to original graphs. This allows one transfer to finiteness results from the Booth-Lueker ideals to other edge ideals, if there are some. However in this setting of restricting the number of vertices on the graph one necessarily has only finitely many possible edge ideals that we can get and hence syzygies by default are finitely generated.

We would like to consider sequences of cellular resolutions coming from graphs where we have not restricted the number of vertices, which means we would have to work over the polynomial ring with infinitely many variables. In practice we can consider the family where each cellular resolution has its own polynomial ring in suitably many variables.

\section{Unrestricted families of cellular resolutions}
In this section we want to consider families of cellular resolutions that do not have restrictions on the number of variables appearing in the polynomial ring. By our previous definition of family of cellular resolutions this means we would have to work over the polynomial ring with infinitely many variables. However, each individual cellular resolution we consider is still assumed to only have finitely many variables in the defining ideal, that is it lives in a polynomial ring with finitely many variables. This observation allows us to consider the unrestricted family as something consisting of cellular resolutions each over their own polynomial ring. 

For this section we assume all our polynomial rings are over the same base field $k$.

\begin{proposition}
If $F$ is a cellular resolutions over the polynomial ring $S=k[x_1,\ldots,x_n]$ then the change of variables to the ring $S'=k[y_1,\ldots,y_m]$, with $m\geq n$, gives also a cellular resolution. 
\end{proposition}
\begin{proof}
Relabelling the variables does not change the relations between the monomials, hence the cell complex supports a cellular resolution after the change of the ring. 
\end{proof}
\begin{remark}
There is often more than one way of changing the variables to the bigger ring, however not all of the possible changes work with the cellular resolution maps that we want to have. 
\end{remark}

There are multiple options how to change the variables to go from one ring to another. However our aim is to study families of cellular resolutions in this setting so we are only interest in those changes of variable that are compatible with a cellular resolution map. To illustrate this we have the following example.

\begin{example}
Let $S_1=k[x,y,z,w]$ and $S_2=k[x,y,z,w,t]$ be two polynomial rings. Let us consider the cellular resolutions for the ideals $I_{P_4}$ and $I_{P_5}$. From Section \ref{paths} we know that they both have a minimal cellular resolution. The minimal resolution for $I_{P_4}$ is the resolution of $G_1$ from Example \ref{exgraph} considered over the ring $S_1$. This resolution is supported on a cell complex formed on three vertices and two edges. We can compute a minimal resolution for $I_{P_5}$ over the ring $S_2$:
$$S\xleftarrow{\left[ \begin{array}{cccc} 
xy&yz& zw& wt
\end{array}\right]} S^4\xleftarrow{\left[\begin{array}{cccc}
-z&-wt&0&0\\
x&0&-w&0\\
0&0&y&-t\\
0&xy&0&z
\end{array}\right]}^4\xleftarrow{\left[\begin{array}{c}
wt\\ -z\\ x\\ xy
\end{array}\right]} S\leftarrow 0.$$
Changing the ring for the resolution of $I_{P_4}$ from $S_1$ to $S_2$ has multiple ways to do it if we allow all possible changes and permutations. We want to be able to map the resolution to that of $I_{P_5}$, and so we only want those changes that give us a subset of the generators of $I_{P_5}$. In practise the possible subsets that can give a map between the resolutions are $xy,yz,zw$ and $yz,zw,wt$. The changes of variables that give these are $x\mapsto x,y\mapsto y,z\mapsto z,w\mapsto w$ and $x\mapsto y,y\mapsto z,z\mapsto w,w\mapsto t$ and $x\mapsto t,y\mapsto w,z\mapsto z,w\mapsto y$ and $x\mapsto w,y\mapsto z,z\mapsto y,w\mapsto x$, and they give maps that are embeddings between the cellular resolutions with possibly a change of orientation on the cells. 
\end{example}

The example above then motivates the following definition for a family of cellular resolution with no restriction on the polynomial ring. 
\begin{definition}
Let $\F$ be family of cellular resolutions such that each resolution $F_i$ is over a polynomial ring $S_i$. We call such family the \emph{unrestricted family of cellular resolutions}. 

The unrestricted family forms a category with the objects being the individual resolutions and morphisms are compositions of a change of a ring map and a cellular resolution morphism. 
\end{definition}

We can lift many definitions from the earlier sections to the unrestricted family setting. This includes the definition of covering in Definition \ref{coveringdef}.
It still holds since the cell complexes do not change and morphisms behave like the maps between cell complexes in the ordinary family case.
Furthermore Definition \ref{linearfam} of a linear family can also be applied directly to the unrestricted family.


Next we want to consider the representations of the unrestricted family. Let $\infring$ denote the polynomial ring with infinitely many variables. We want to consider representations to the modules over this infinite ring. 
\begin{definition}
A representation of an unrestricted family $\F$ is a functor $M:\F\rightarrow \infmod$.
\end{definition}
Then we can define the following particular representations.

\begin{definition}
The $t$-th module representation
$$\submodinf: \F\rightarrow \infmod$$
 such that $\submodinf(F_i)$ is the free module over $\infring$ with generators in the same degrees as the $t$-th free module in the resolution, and $\submodinf(F_i\rightarrow F_j)$ is the matrix of the chain map from $F_i$ to $F_j$ on the $t$-th component.
\end{definition}
\begin{definition}
\label{infsyzf}
Let $\F$ be an unrestricted family of cellular resolutions. The $t$-th syzygy functor 
$$\syzmodinf:\F\rightarrow \infmod$$
is defined by taking $F\in \F$ to the finitely generated submodule of a free $\infring$-module $\submodinf(F_i)$ with the same generators as the $t$-th syzygy module of $F$, and the morphisms are restrictions of the free module maps to the submodule. 
\end{definition}
\noindent
Note that $\syzmodinf$ is a subfunctor of $\submodinf$.

Making use of the fact that the unrestricted family behaves similar to the family over a single polynomial ring we get analogues of the results from the previous sections in the setting of unrestricted families.
\begin{proposition}
\label{infnoeth}
If $\F$ is an unrestricted family of cellular resolutions such that it is linear then the representations $\operatorname{Rep}_{\infring}(\F)$ form a noetherian category. 
\end{proposition}
\begin{proof}
Suppose that $\F$ is an unrestricted family of cellular resolutions such that it is linear. Then the morphisms in this family behave in the same way as for the linear family over a single polynomial ring. Thus we have that the principal projectives are all noetherian following the proof of Proposition \ref{noethrep}. Then it follows from the principal projectives being noetherian that the representation category $\operatorname{Rep}_{\infring}(\F)$ is noetherian.
\end{proof}
\begin{proposition}
\label{inffg}
If $\F$ is an unrestricted family of cellular resolutions such that it is linear and the cell complexes have covering in dimension $t$ for all $i$ large enough, then the syzygy functor $\syzmodinf$ is finitely generated for $t$.
\end{proposition}
\begin{proof}
The generators of the modules in $\submodinf(F)$ correspond to the cells in the cell complex supporting the resolution of $F$. Therefore we can use the same argument as in the proof of Lemma \ref{covering}. Assume we have $t$-covering for some large enough $i$. Then we have that every cell complex above $i$ in the family is covered by some finite set of cell complexes. That means there is a finite set of $t$-cells that cover all other $t$-cells. This implies that on the level of free modules, every generator in the $t$-th modules is reachable from a finite set of generators since the chosen maps correspond to the cellular maps. Then by definition the representation $\submodinf$ is finitely generated for $t$. Moreover by Proposition \ref{infnoeth} the linearity implies that representations are noetherian, in particular this means any subrepresentation of finitely generated representation is finitely generated, so we have that $\syzmodinf$ is finitely generated. 
\end{proof}

\subsection{Simplexes supporting resolutions}
\label{dimproblem}
In this section we want to consider an explicit example of an unrestricted family of cellular resolutions. 
Let us take the $n$-simplex with variable labels where we add a new one variable at each dimension. Each simplex supports a cellular resolution but over a different ring.  This will give us a family of minimal cellular resolutions $\F$ and denote the resolution corresponding to the $n$-simplex by $F_n$.  
This family of cell complexes is shown in Figure \ref{scfamily}.

\begin{figure}
\begin{center}
\begin{tikzpicture}
\filldraw[color=red, fill=red!5,  thick](1,0)--(1.5,2);
\filldraw[color=red, fill=red!5,  thick](2.5,0)--(4.5,0)--(3.5,2)--cycle;
\filldraw[color=red, fill=red!5,  thick](5.5,0)--(7.5,0)--(6.5,2)--cycle;
\filldraw[color=red, fill=red!5,  thick](8,0.9)--(7.5,0)--(6.5,2)--cycle;
\draw[color=red, dashed](8,0.9)--(5.5,0);

\filldraw [black](1,0) circle (2pt) node[anchor=north] {$x_1$};
\filldraw [black](1.5,2) circle (2pt) node[anchor=south] {$x_2$};
\filldraw [black](2.5,0) circle (2pt) node[anchor=north] {$x_1$};
\filldraw [black](4.5,0) circle (2pt) node[anchor=north] {$x_3$};
\filldraw [black](3.5,2) circle (2pt) node[anchor=south] {$x_2$};
\filldraw [black](5.5,0) circle (2pt) node[anchor=north] {$x_1$};
\filldraw [black](7.5,0) circle (2pt) node[anchor=north] {$x_3$};
\filldraw [black](6.5,2) circle (2pt) node[anchor=south] {$x_2$};
\filldraw [black](8,0.9) circle (2pt) node[anchor=west] {$x_4$};
\end{tikzpicture}
\end{center}
\caption{Labelled cell complexes for the family in Section \ref{dimproblem}.}
\label{scfamily}
\end{figure}
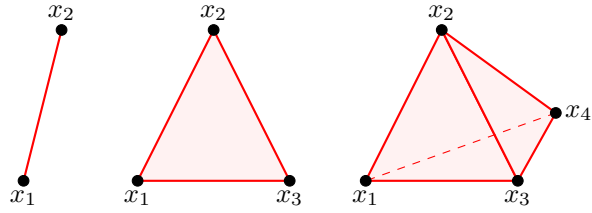
For the morphisms in the family we can take all possible maps, note that between $F_n$ and $F_{n+1}$ there are $(n+1)!$ maps. However if we are only interested in which cells map to which cell, then it is enough to consider only the maps that come from choosing $n$ variables from $n+1$ variables. 
\begin{proposition}
For a fixed $t$ the representation $\syzmodinf$ of $\F$ is finitely generated.
\end{proposition}
\begin{proof}
Let $\F$ be the family of cellular resolutions supported on the simplices. Then any map between some resolutions $F_i$ and $F_j$ can be written as composition of consecutive maps, and we have map between any two of the cellular resolutions. Therefore the family $\F$ is a linear family and we get that the representation category is noetherian. 

Next we want to show that for a fixed dimension $t$ we have covering of $t$ dimensional cells. For this we will only consider the maps that give different set of variables after the change of ring. Then we will have $n+1$ maps between the resolutions $F_n$ and $F_{n+1}$. Let $X_n$ denote the $n$-simplex. The resolution maps we are considering correspond to the embeddings of $X_n$ to $X_{n+1}$. There are $n+1$ possible embeddigns between $X_n$ and $X_{n+1}$, which also cover all $n$-dimensional cells. However, none of the embeddings cover the $(n+1)$-dimensional cell. Thus if we fix a dimension $t$, taking all cell complexes up $X_t$ will then give a covering of the $t$-dimensional cells with the embeddings. Thus for the family $\F$ we have that for a fixed $t$ we have $t$-covering. 

Hence by Proposition \ref{inffg} we get that the representation $\syzmodinf$ is finitely generated for a fixed $t$.
\end{proof}

\section{Open questions}
Finally we list few open questions that arise from the previous sections. 

In all of our examples the families of cellular resolutions have been linear, so a natural question would be to ask what about non-linear families. Can we find non-linear families that have finitely generated syzygies or satisfy other properties like noetherian representation category?
Another observation in all our examples is that we used noetherianity to prove finite generation of syzygies, and often the noetherianity of the representation category is inherited from the nice structure the family has which in first place suggested finite generation of syzygies. One can then ask whether there exists a family of cellular resolutions that has finitely generated syzygies but without a noetherian representation category?

In this paper we focused on the syzygies of the families of cellular resolutions. Thus one can ask if the representations can be used to study other properties than syzygies for the families. We also note that the Gr\"obner property was used to study the families and this could be an interesting direction to look at. Moreover, the paper of Sam and Snowden \cite{SS} contains other structures, like lingual structures, that have not been addressed in this paper. One possible question is do the lingual structures have particular meaning or application with cellular resolutions and can we find families that satisfy the conditions to have these structures.

As a last open question we pose further work on the unrestricted case of families. The approach we have proposed could be considered the naive way to deal with the requirement of having cellular resolutions over different rings, and we have not dwelt very deeply into it.  The theory of modules over polynomial ring with infinitely many variables  could offer tools to work further with the proposed setting and also make use of different representations for these families. Another direction to take with these are the cases where there is polynomial ring with a maximal number of variables, in which case the modules can be taken over that. This mixed finite case also allows different permutations of variables within the same ring, which are not morphisms of cellular resolutions in the fixed ring case.



\begin{thebibliography}{99} 
\bibitem{cf}
Thomas Church and Benson Farb.
Representation theory and homological stability.
\emph{Adv. Math.}{\bf 245} (2013), 250–314. 

\bibitem{bps}
Dave Bayer, Irena Peeva, and Bernd Sturmfels.
Monomial resolutions.
\emph{Math. Res. Lett.} {\bf 5} (1998), no. 1-2, 31-46.

\bibitem{def}
Dave Bayer and Bernd Sturmfels.
Cellular Resolutions of Monomial Modules.
\emph{J. Reine Angew. Math.} {\bf 502} (1998), 123-140.






\bibitem{af}
Anton Dochtermann and Fatemeh Mohammadi.
Cellular resolutions from mapping cones. 
\emph{J. Combin. Theory Ser. A} {\bf 128} (2014), 180–206.

\bibitem{emo}
John Eagon, Ezra Miller, and Erika Ordog.
Minimal resolutions of monomial ideals.
arXiv:1906.08837.

\bibitem{gsyz}
David Eisenbud.
\emph{Geometry of Syzygies.}
Graduate Texts in Mathematics, 229. Springer-Verlag, New York, 2005. xvi+243 pp.

\bibitem{bs}
Alexander Engstr\"om, Laura Jakobsson and Milo Orlich.
Explicit Boij-S\"oderberg theory of ideals from a graph isomorphism reduction. arXiv:1810.10055.


\bibitem{E10}
Alexander Engstr\"om and Patrik Nor\'{e}n.
Cellular resolutions of powers of monomial ideals.
arXiv:1212.2146



\bibitem{ht}
J\"urgen Herzog and Yukihide Takayama.
Resolutions by mapping cones. (English summary) 
The Roos Festschrift volume, 2. 
\emph{Homology Homotopy Appl.} {\bf 4} (2002), no. 2, part 2, 277-294. 

\bibitem{me}
Laura Jakobsson.
The category of cellular resolutions.
arXiv:1902.09873.




\bibitem{ml}
Saunders Mac Lane.
\emph{Categories for the Working Mathematician, second edition.}
Graduate Texts in Mathematics, 5. Springer-Verlag, New York, 1997. 314pp.


\bibitem{cca}
Ezra Miller and Bernd Sturmfels.
\emph{Combinatorial Commutative Algebra.}
Graduate Texts in Mathematics, 227. Springer-Verlag, New York, 2005. 417 pp.


\bibitem{SS}
Steven Sam and Andrew Snowden.
Gr\"obner methods for representations of combinatorial categories. 
\emph{J. Amer. Math. Soc.} {\bf 30}(2017), no. 1, 159-203.








\end{thebibliography}
\end{document}